\documentclass[11pt]{amsart}
\usepackage[T1]{fontenc}
\usepackage{lmodern}
\usepackage[margin=1in]{geometry}
\usepackage{microtype}
\usepackage{tabularx}
\usepackage{xcolor}
\usepackage[inline]{enumitem}
\usepackage[authoryear,round]{natbib}
\usepackage{hyperref}
\usepackage{cxcase}
\cxtheorem{aim-problem-35}{thm:common-A}
\cxverbatim{aim-problem-35}
\cxtheorem{aim-problem-38}{thm:pot}
\cxverbatim{aim-problem-38}
\cxtheorem{aim-problem-36}{thm:aim36}
\cxverbatim{aim-problem-36}
\cxtheorem{aim-problem-37}{thm:aim37}
\cxverbatim{aim-problem-37}

\cxtheorem{macdonald-schur-convexity}{thm:macdonald}
\cxverbatim{macdonald-schur-convexity}

\cxtheorem{theta-derivative-log-concavity}{thm:theta}
\cxverbatim{theta-derivative-log-concavity}

\cxtheorem{variance-only-matrix-discrepancy}{thm:discrepancy}
\cxverbatim{variance-only-matrix-discrepancy}

\cxtheorem{mixed-norm-general-s}{thm:mixed-general}
\cxverbatim{mixed-norm-general-s}
\cxtheorem{yufei-psd}{thm:mixed}
\cxparaphrase{yufei-psd}

\cxtheorem{quantum-coupon-collector}{thm:quantum-coupon-collector}
\cxverbatim{quantum-coupon-collector}

\cxtheorem{odonnell-matrix-conjecture}{thm:odonnell-matrix-conjecture}
\cxparaphrase{odonnell-matrix-conjecture}

\cxtheorem{courtade-volume-conjecture}{thm:courtade-volume-conjecture}
\cxverbatim{courtade-volume-conjecture}

\cxtheorem{dpp-feasible-step}{thm:dpp}
\cxverbatim{dpp-feasible-step}

\cxtheorem{log-volume-distance}{thm:volume-distance}
\cxparaphrase{log-volume-distance}
\cxtheorem{lorentzian-jensen}{thm:lorentzian}
\cxparaphrase{lorentzian-jensen}

\cxtheorem{sdd-nystrom-diminishing-returns}{thm:sdd-nystrom-diminishing-returns}
\cxverbatim{sdd-nystrom-diminishing-returns}

\cxtheorem{osi-sketch-and-solve}{thm:osi-sketch-and-solve}
\cxverbatim{osi-sketch-and-solve}

\cxtheorem{hamiltonian-nepv-identity}{thm:hamiltonian-nepv-identity}
\cxverbatim{hamiltonian-nepv-identity}

\cxtheorem{qrcp-orthonormal-greedy}{thm:qrcp-orthonormal-greedy}
\cxverbatim{qrcp-orthonormal-greedy}
\usepackage{listings}
\usepackage{graphicx}
\usepackage{url}
\usepackage{multirow}

\AtBeginDocument{\let\cite\citep}
\hypersetup{colorlinks=true,linkcolor=blue!45!black,citecolor=blue!45!black,urlcolor=red!65!black}
\setlist{nosep,leftmargin=*}
\makeatletter
\patchcmd{\section}{\scshape}{\fontfamily{ppl}\bfseries\scshape\selectfont}{}{%
  \PackageWarning{main}{Could not bold the section font; amsart may have changed}}
\makeatother

\numberwithin{equation}{section}   %

\title{GPT, the Counterexample Machine}
\author[Count ex Machina]{Count ex Machina\textsuperscript{\normalfont*}}
\thanks{*\,Project name used to future proof against possible multiple contributors. The first version is authored by Suvrit Sra (\texttt{s.sra@tum.de}; TU Munich, Dept.~of Math, School of CIT.)}
\date{Draft of August 9, 2026~~(see {\fontfamily{ppl}\selectfont\itshape\href{https://github.com/suvrit/count-ex-machina}{https://github.com/suvrit/count-ex-machina}}~for the whole timeline)}

\begin{document}
\begin{abstract}
  This document reports over 15 counterexamples found using GPT Pro over the course of 12 months, for problems in combinatorics, number theory, convexity, analysis, and other areas. The counterexamples concern both (reasonably) well-known problems as well as lesser known instances. The counterexamples are accompanied by the repo \href{https://github.com/suvrit/count-ex-machina}{\fontfamily{ppl}\selectfont\itshape https://github.com/suvrit/count-ex-machina}, and we warmly welcome new contributions toward helping building a rich \emph{``open-source countex library.''}
\end{abstract}
\maketitle

% ==== tex/01-literature.tex ====
\section{Introduction}
\noindent AI assisted mathematics has literally exploded; as of writing this document, the website \href{https://aimath.robertj1.com}{[AIMath]} lists 438 as of Aug.~9, 2026 (and 506 as of Aug.~25, 2026!) proofs and disproofs of a variety of mathematical questions. Our focus in this manuscript is on counterexamples that refute a conjecture or provide a disproof of some problem, providing both a human friendly proof and a machine-verifiable certificate.

\vskip5pt
\noindent\textbf{\emph{Why counterxamples?}} The typical asymmetry between proving a claim (universal quantification) versus disproving (existence of one counterexample) makes it particularly attractive to prioritize AI-assisted search. And available AI assistants have been strikingly successful in 2026 at constructing counterexample. Ten noteworthy conjectured claimed to be disproved by AI in 2026 include:
\vskip5pt
\begin{small}
  \begin{enumerate}[label=(\roman*),itemsep=0pt]
  \item the Jacobian conjecture in 3 dimensions by an explicit Keller map~\citep{Alpoge2026}; %
  \item the Erd\H{o}s unit-distance conjecture, disproved at~\citet{OpenAIUnitDistance2026}; with an explicit exponent by \citet{Sawin2026}, and a human-verified digest by \citet{UnitDistanceRemarks2026};
  \item the Gaussian completely monotone conjecture \citep{GuSellke2026};
  \item conjectured log-convexity of Fisher information along heat flow \citep{ZouEtAl2026};
  \item Litvak's Gaussian-minima conjecture \citep{Kunisky2026};
  \item the pairwise-independent correlation gap \citep{RamachandraNatarajan2026};
  \item near-quadratic Elekes--R\'onyai expansion \citep{Liu2026};
  \item the Gaussian moments conjecture \citep{LongGaussian2026};
  \item the $xz$- and Mathieu conjectures for $SU(2)$ \citep{LongSU22026}; and
  \item conjectured Schur positivity of claw-free chromatic symmetric functions
    \citep{MatherneMorales2026}.
  \end{enumerate}
\end{small}

\vskip5pt
But refutation is not the only aspect the models have been exemplary at: the document~\citep{OpenAITenAdvances2026}  includes 10 results obtained by OpenAI's internal Astra model, with several results accompanied by machine-checkable Lean certificates.\footnote{Allegedly, as of this writing, the claimed disproof of the Connes rigidity conjecture is invalid, because the AI system silently used a group with an infinite center, which violates the conjecture's framing.}
And the noted collection \href{https://aimath.robertj1.com}{[AIMath]} lists several successful proofs. Though a more mixed set of results from problems designed to challenge AI approaches is in the  ``First Proof'' \citep{abouzaid2026proof} project, where it seems so that AI has not yet managed to make breakthroughs. Continuing the topic of conjectures and their mechanical verification, this new project seems like a wonderful resource~\citep{FormalConjecturesPaper2026}.

\vskip5pt
\noindent\emph{\textbf{\textcolor{darkblue}{Our Philosophy.}}} First, a short word on our \emph{raison d'être}. This work started with the author's effort in collecting some of the AI-assisted disproofs we found over the past 12 months. Importantly, we  wanted our exposition of the collected results to be human-readable and human-verifiable, \textbf{\emph{while honestly crediting AI for the heavy-lifting}};\footnote{For instance, it is quite unclear from numerous new (clearly AI-assisted) works appearing on the arXiv whether humans are fairly acknowledging AI's role; our repository is dedicated to collecting results that can be largely attributed to AI, though without trying to be either philosophic or pedantic about ``who did what?''}. In some cases, we share refinements that further enrich the discussion. We supplement the exposition with source code that can be used to machine-verify counterexamples. Though ultimately, it is for the joy of understanding mathematics, of proving and disproving claims, and of sharing progress with others that underscores our motivation; the machine-checkable component is there to support this ideal, beyond mere pedantic necessity.

% ==== end tex/01-literature.tex ====
% ==== tex/02-admission.tex ====
\section{Summary of the counterexamples and resolutions}
Some simple criteria for including counterexamples emerged when preparing this manuscript. We hope to use these criteria as a rough yardstick for gauging new contributions to the collection. In particular, we propose the following simple criteria:
\begin{definition}[\textbf{\textcolor{darkblue}{Admission criteria}}]
\begin{enumerate}
\item The refuted statement appeared publicly before the counterexample was sought (e.g., in a paper or preprint, a problem list, a talk, or in a public post, etc.). The public appearance condition excludes statements invented in order to be refuted.
\item The original hypotheses and quantifiers are fully preserved.  Counterexample to a strengthening or any modification of the original problem are not admitted as refutations.\footnote{It may well be that a modified statement is the more ``natural'' one, but given the automation involved, to be pedantically correct, we cannot admit any changes, unless the modified version itself satisfies condition~1.}
\item The refutation is independently verifiable, through exact arithmetic, an interval certificate, an analytic calculation, proof assistant, etc.
\end{enumerate}
\end{definition}

\noindent\textbf{Note about problems.} Most of the statements studied in the current manuscript involve problems that are likely well-known to specialists but not well-known more broadly. Our collection is open to everybody for contributions, as long as the contributions meet the minimalistic criteria listed above. Moreover, we neither enforce nor advocate any particular verifier, but we do expect every contributor to ensure verification. This manuscript is accompanied by a supporting git repo: \href{https://github.com/suvrit/count-ex-machina}{\fontfamily{ppl}\selectfont\itshape https://github.com/suvrit/count-ex-machina}, which is a living repo and will be the definitive resource on any bugfixes, improvements, additions or otherwise, and we welcome contributions. %

% ==== end tex/02-admission.tex ====

% ==== tex/ledger.tex ====

\newlength{\cxnumwd}\setlength{\cxnumwd}{2.1em}
\newlength{\cxnumsep}\setlength{\cxnumsep}{0.65em}
\setlength{\LTleft}{-\dimexpr\cxnumwd+\cxnumsep\relax}
\setlength{\LTright}{\fill}
\setcounter{cxrow}{0}   %
\begin{longtable}{@{}>{\raggedleft\arraybackslash}p{\cxnumwd}@{\hspace{\cxnumsep}}p{0.46\textwidth}p{0.50\textwidth}@{}}
  \caption{Summary of counterexamples. \ec{} denotes \emph{Exact Certificate}; \ic{} denotes \emph{Interval Certificate}.}\label{tab:ledger}\\
  \cmidrule[\heavyrulewidth]{2-3}
  & Problem statement & Certificate(s) \\
  \cmidrule{2-3}
  \endfirsthead
  \cmidrule[\heavyrulewidth]{2-3}
  & Problem statement & Certificate(s) \\
  \cmidrule{2-3}
  \endhead
  \cxn & Borcea--Br\"and\'en Problem 35: existence of Johnson polynomial representation for stable polynomials \citep{AIMStableProblems} & Explicit polynomial (deg-5) with immanantal obstruction.\ \ec; human proof Thm.~\ref{thm:common-A}.\\
  \cxn & Borcea--Br\"and\'en Problem 38: Stable polynomial Permanent-On-Top \citep{AIMStableProblems} & Exact Schur-basis expansion; four violating partitions.\ \ec; human proof Thm.~\ref{thm:pot}.\\
  \cxn & Borcea--Br\"and\'en Problem 36: Schur positivity of symmetric stable polynomials \citeyearpar{AIMStableProblems} & Determinantal polynomial in Schur-basis with $[s_{(1^5)}]q=-2972$.\ \ec; human proof Thm.~\ref{thm:aim36}.\\
  \cxn & Borcea--Br\"and\'en Problem 37: stable Schur inequality (lower endpoint) \citeyearpar{AIMStableProblems} & Schur-\emph{positive} member of the same family with $a_{(1^5)}<f^{(1^5)}a_{(5)}$.\ \ec; human proof Thm.~\ref{thm:aim37}.\\
  \cxn & \citet[][Theorem 2.1]{McSwiggenSahi2026}: Claimed log- and Schur-convexity of a ratio of Macdonald polynomials & Failure of Schur (and thus also log-) convexity (rank 2).\ \ec; human proof Thm.~\ref{thm:macdonald}.\\
  \cxn & \citep[][Conj.~2.5]{coffey2013log} Turán inequalities for Jacobi-Theta derivatives $\Phi^{(n)}$ & Rigorous interval evaluation for $n=9$.\ \ic; (reduction in Thm.~\ref{thm:theta}).\\
  \cxn & Variance-sensitive Matrix Spencer~\cite[Prob.~4.25]{Bandeira2016} & Diagonal family with discrepancy $\Theta(\sqrt{\log n})$.\ \ec; human proof Thm.~\ref{thm:discrepancy}.\\
  \cxn & Mixed-norm $\|A^TB\|_{p,q}^2\le \|A^TA\|_{p,p}\|B^TB\|_{q,q}$ for $1\le p\le q$~\citep{SahSawhneyStonerZhao2020} & Explicit analytic disproof; \ec; human proof Thm.~\ref{thm:mixed-general}.\\
  \cxn & The same question at $B=A$: $\|Z\|_{s,q}^2\le\|Z\|_{s,s}\|Z\|_{q,q}$ for completely positive $Z$ & Explicit disproof via a rank-2 $21\times 21$ PSD matrix;\ \ic; no closed-form disproof;~\ref{thm:mixed}).\\
  \cxn & ``Quantum'' coupon-collector: $\sum_{\emptyset\ne S}(-1)^{|S|-1}X_S^{-1}\stackrel{?}{\succ} 0$~\citep{Sra2017QuantumCoupon}. & Disproof via rational $3\times3$ matrices with $n=6$;\ \ec; human proof Thm.~\ref{thm:quantum-coupon-collector}.\\
  \cxn & \citet{ODonnell2015AlmostDiagonal} $\|A-D\|_1^2 \stackrel{?}{\le}C\|\lambda(A)-D\|_1$, where $A$ is PSD and $D=\Diag(A)$. & Givens-rotation family with a rational lower-bound; \ec; human proof Thm.~\ref{thm:odonnell-matrix-conjecture}.\\

  \cmidrule{2-3}\\
  & \multicolumn{2}{c}{\textsc{Additional Examples (relegated to the Appendix)}}\\[5pt]
  \cmidrule{2-3}\\

  \cxn & \citet{Courtade2017Concavity}: $V(K{+}L{+}B)^{1/d}V(B)^{1/d}+V(K)^{1/d}V(L)^{1/d}\le V(K{+}B)^{1/d}V(L{+}B)^{1/d}$. & Disproof in $\R^4$: zonotope plus two segments; \ec; human proof Thm.~\ref{thm:courtade-volume-conjecture}.\\
  \cxn & \citet{MarietSra2015}'s feasible-step for a Picard iteration for DPP likelihood ascent.  & Exact $2\times2$ feasible update violating likelihood ascent.\ \ec; human proof Thm.~\ref{thm:dpp}.\\
  \cxn & The divergence $\log\frac{V[(X+Y)/2]}{\sqrt{V(X)V(Y)}}$ is a squared distance for convex bodies~\citep{sra2013}. & Cubic realizable as a volume polynomial; violates triangle inequality.\ \ec; human proof Thm.~\ref{thm:volume-distance}.\\
  \cxn & Lorentzian Jensen--Bregman metrics: $\sqrt{\J_G}$ a semimetric for all homogeneous Lorentzian $G$. & Disproof via same cubic viewed as a Lorentzian polynomial.\ \ec; human proof Cor.~\ref{thm:lorentzian}.\\
  \cxn & An SDD matrix has a marginal Nystr\"om error reduction\citep[Prob.~4.6(b)]{AmselEtAl2026}.  & Disproof via an integer $4\times4$ matrix.\ \ec; human proof Thm.~\ref{thm:sdd-nystrom-diminishing-returns}.\\
  \cmidrule[\heavyrulewidth]{2-3}
\end{longtable}
\setlength{\LTleft}{\fill}  %
% ==== end tex/ledger.tex ====

% ==== tex/cases.tex ====

% ==== counterexamples/aim-problems/case.tex ====

\cxtitle{Borcea--Br\"and\'en AIM problems}

\cxsubtitle{Problems 35 and 38: mixed-determinant representation and Permanent-On-Top}

\begin{cxcredits}
\posedby{Borcea and Br\"and\'en \citeyearpar{AIMStableProblems}}
\foundby{GPT-5 (Pro)}{2026-01-10}
\formalizedby{S.~Sra}
\auditedby{S.~Sra}
\contributedby{S.~Sra}
\end{cxcredits}

\begin{cxcontext}
A polynomial $f\in\R[z_1,\ldots,z_n]$ is \emph{real stable} if $f(z_1,\ldots,z_n)\ne0$ whenever every $z_i$ has strictly positive imaginary part.  A homogeneous symmetric polynomial of degree $d$ expands uniquely in the basis of Schur polynomials as $p=\sum_{\lambda\vdash d}a_\lambda s_\lambda$; we say $p$ is \emph{Schur-positive} if every $a_\lambda\ge0$.  For a partition $\lambda\vdash d$ we write $f^\lambda$ for the number of standard Young tableaux of shape $\lambda$, $\lambda'$ for the conjugate partition, and, for a $d\times d$ matrix $A=(a_{ij})$,
\[
\Imm_\lambda(A) :=\sum_{\sigma\in S_d}\chi^\lambda(\sigma)\prod_{i=1}^da_{i\sigma(i)},
\]
for the immanant induced by the irreducible character $\chi^\lambda$ of $S_d$. The two extreme cases are $\Imm_{(1^d)}=\det$ and $\Imm_{(d)}=\per$.  Next, for $n\times n$ matrices $A_1,\ldots,A_m$ the \emph{mixed-determinant} is
\[
\eta(A_1,\ldots,A_m) :=\sum_{(S_1,\ldots,S_m)}\det(A_1[S_1])\cdots\det(A_m[S_m]),
\]
where the sum runs over all ordered partitions of $\{1,\ldots,n\}$ into disjoint sets $S_1,\ldots,S_m$, and $A[S]$ denotes the principal submatrix of $A$ with rows and columns in $S$.  When $A\succeq0$ the polynomial $\eta(z_1A,\ldots,z_nA)$ is symmetric, homogeneous, real stable, and has the (Schur-positive) representation
\begin{equation*}
  \eta(z_1A,\ldots,z_nA)=\sum_{\lambda\vdash n}\Imm_{\lambda'}(A)s_\lambda(z_1,\ldots,z_n).
\end{equation*}
The two problems below ask how far those properties go the other way.
\end{cxcontext}

\begin{cxsource}{aim-problem-35}
Borcea--Br\"and\'en, AIM problem list \citeyearpar{AIMStableProblems}, Problem 35
\end{cxsource}

\begin{cxstatement}{aim-problem-35}
Characterize all symmetric, real stable and homogeneous polynomial of degree $d$ in $n$ variables.  Is it in fact so that if $p$ is such a polynomial, then there is a $d\times d$ positive semidefinite $d\times d$ matrix $A$ such that
\[p(z_1,\ldots,z_n)=\eta(z_1A,\ldots,z_nA)?\]
This is not obvious even for $n=2$.
\end{cxstatement}

\begin{cxsummary}{aim-problem-35}
Borcea--Br\"and\'en Problem 35: common-matrix mixed-determinant representation \citeyearpar{AIMStableProblems}
\end{cxsummary}

\begin{cxcertificate}{aim-problem-35}
Stable degree-five polynomial plus size-five immanant POT obstruction.
\end{cxcertificate}

\begin{cxsource}{aim-problem-38}
Borcea--Br\"and\'en, AIM problem list \citeyearpar{AIMStableProblems}, Problem 38
\end{cxsource}

\begin{cxstatement}{aim-problem-38}
Lieb's Permanent-On-Top (POT) Conjecture for the symmetric group on $d$ elements [J,L] asserts that if $A$ is a $d\times d$ positive semi-definite matrix then $\Imm_\lambda(A)\le f^\lambda\per(A)$ for all $\lambda\vdash d$.  An affirmative answer to the following problem would in particular imply the validity of the POT Conjecture.

\emph{Problem 38.}  Let $p$ be as in Problem 36.  Is $a_\lambda\le f^\lambda a_{(1^d)}$ for all $\lambda\vdash d$?
\end{cxstatement}

\begin{cxsummary}{aim-problem-38}
Borcea--Br\"and\'en Problem 38: Stable Permanent-On-Top \citeyearpar{AIMStableProblems}
\end{cxsummary}

\begin{cxcertificate}{aim-problem-38}
Exact Schur expansion; four violating partitions.
\end{cxcertificate}

\begin{cxrefutation}
We answers to the above problems are negative. Let $S=x_1+\cdots+x_5$ and define the polynomial
\begin{equation}\label{eq:stable-p}
p(x_1,\ldots,x_5)=\prod_{i=1}^5\Bigl(x_i+5\nlsum_{j\ne i}x_j\Bigr)=\prod_{i=1}^5(5S-4x_i).
\end{equation}
Every factor has strictly positive coefficients, hence has positive imaginary part whenever all variables do; therefore $p$ is real stable. Moreover, $p$ is symmetric, homogeneous of degree five in five variables, and has strictly positive monomial coefficients. %

\begin{theorem}[\statusfalse: Stable Permanent-On-Top, AIM Problem 38]\label{thm:pot}
The polynomial \eqref{eq:stable-p} satisfies all hypotheses of Problem~\ref{ctx:aim-problem-38} but violates $a_\lambda\le f^\lambda a_{(1^5)}$.
\end{theorem}
\begin{proof}
A short explicit calculation shows that $p$ has the Schur expansion
\begin{align*}
p={}&625s_{(5)}+4500s_{(4,1)}+7125s_{(3,2)}+8150s_{(3,1,1)}+7525s_{(2,2,1)}+6580s_{(2,1,1,1)}+1281s_{(1^5)}.
\end{align*}
For $\lambda=(3,2)$, the hook-length formula gives $f^{(3,2)}=5$, whereas
\[
7125>5\cdot1281=6405.
\]
The same upper endpoint fails for $(3,1,1)$, $(2,2,1)$, and $(2,1,1,1)$.  The accompanying source package reconstructs the monomial expansion and inverts the Kostka matrix over $\mathbb Z$.
\end{proof}

The polynomial also has a diagonal positive-definite mixed-determinant representation.  For $j=1,\ldots,5$, let $D_j$ denote $5I$ except for the $j$-th entry being just 1.
Then the multivariate Johnson polynomial satisfies
\[
\eta(x_1D_1,\ldots,x_5D_5)=p(x_1,\ldots,x_5).
\]

\begin{theorem}[\statusfalse: common-matrix representation, AIM Problem 35]\label{thm:common-A}
There is no $5\times5$ Hermitian positive semidefinite matrix $A$ such that $p(x_1,\ldots,x_5)=\eta(x_1A,\ldots,x_5A)$.
\end{theorem}
\begin{proof}
Since $\eta(x_1A,\ldots,x_5A)$ is Schur positive with immanants as coefficients, the representation $p(x)=\eta(x)$ would identify the Schur coefficients (up to normalization). Since Lieb's permanental dominance inequality is known to hold in size five (indeed in size $\le 13$~\cite{Wanless2022} it is known), it would force
$a_\lambda\le f^\lambda a_{(1^5)}$ for every $\lambda\vdash5$. But this contradictions Theorem~\ref{thm:pot}, hence there cannot be any PSD $A$ for which $p(x)=\eta(x_1A,\ldots,x_5A)$.
\end{proof}

\begin{remark}[Scope]
The example is Schur-positive, and its lower endpoint inequalities hold, so it refutes neither Problem~\ref{ctx:aim-problem-36} nor Problem~\ref{ctx:aim-problem-37}; those require the positive-definite pencils of Theorems~\ref{thm:aim36} and~\ref{thm:aim37} below.  The full independent audit is included as \path{counterexamples/aim-problems/artifacts/}.
\end{remark}
\end{cxrefutation}

\cxsubtitle{Problem 36: stable Schur positivity}

\begin{cxcredits}[aim-problem-36]
\foundby{GPT-5.6 (Pro)}{2026-07-30}
\end{cxcredits}

\begin{cxrefutation}
  We keep the notation of the previous section. %
Problem~\ref{ctx:aim-problem-36} asks whether stability forces a Schur-positive representation. %
\end{cxrefutation}

\begin{cxsource}{aim-problem-36}
Borcea--Br\"and\'en, AIM problem list \citeyearpar{AIMStableProblems}, Problem 36
\end{cxsource}

\begin{cxstatement}{aim-problem-36}
Let $p$ be a symmetric, real stable and homogeneous polynomial of degree $d$ in $n$ variables with $d\le n$ and suppose that $p$ has at least one positive (and then all non-negative) coefficients in the monomial basis.  Is $p$ Schur-positive?  That is, when $p(z_1,\ldots,z_n)$ is expanded in the Schur-basis
\[p(z_1,\ldots,z_n)=\sum_{\lambda\vdash d}a_\lambda s_\lambda(z_1,\ldots,z_n),\]
is $a_\lambda\ge0$ for all $\lambda\vdash d$?
\end{cxstatement}

\begin{cxsummary}{aim-problem-36}
Borcea--Br\"and\'en Problem 36: stable Schur positivity \citeyearpar{AIMStableProblems}
\end{cxsummary}

\begin{cxcertificate}{aim-problem-36}
Positive-definite $5\times5$ det-pencil; exact Schur expansion with $[s_{(1^5)}]q=-2972$.
\end{cxcertificate}

\begin{cxrefutation}
The answer to Problem~\ref{ctx:aim-problem-36} is no already at $d=n=5$.  Let
\begingroup
\scriptsize
\[
J_1=
\begin{pmatrix}
38&19&19&2&36\\
19&38&2&19&36\\
19&2&38&19&36\\
2&19&19&38&36\\
36&36&36&36&72
\end{pmatrix},\quad
J_2=
\begin{pmatrix}
38&19&19&17&21\\
19&38&17&19&21\\
19&17&38&34&36\\
17&19&34&38&36\\
21&21&36&36&42
\end{pmatrix},\quad
J_3=
\begin{pmatrix}
38&34&19&17&36\\
34&38&17&19&36\\
19&17&38&19&21\\
17&19&19&38&21\\
36&36&21&21&42
\end{pmatrix},
\]
\[
J_4=
\begin{pmatrix}
38&4&19&2&21\\
4&8&2&4&6\\
19&2&38&4&21\\
2&4&4&8&6\\
21&6&21&6&42
\end{pmatrix},\qquad
J_5=
\begin{pmatrix}
8&4&4&2&6\\
4&38&2&19&21\\
4&2&8&4&6\\
2&19&4&38&21\\
6&21&6&21&42
\end{pmatrix},
\]
\endgroup
and define the determinantal polynomial
\begin{equation}\label{eq:aim36-q}
q(x_1,\ldots,x_5)=\frac1{648}\det\Bigl(\nlsum_{i=1}^5x_iJ_i\Bigr).
\end{equation}

\begin{theorem}[\statusfalse: stable Schur positivity, AIM Problem 36]\label{thm:aim36}
The polynomial \eqref{eq:aim36-q} is symmetric, homogeneous of degree five in five variables, real stable, and all $126$ of its ordinary monomial coefficients are strictly positive, yet it is \emph{\textbf{not}} Schur positive, since
\[
[s_{(1^5)}]\,q=-2972<0.
\]
\end{theorem}
\begin{proof}
The leading principal minors of $J_i$ are
\[
\begin{array}{c|rrrrr}
 &\Delta_1&\Delta_2&\Delta_3&\Delta_4&\Delta_5\\ \hline
J_1&38&1083&28728&202176&1119744\\
J_2&38&1083&28728&202176&1119744\\
J_3&38&288&8208&202176&1119744\\
J_4&38&288&8208&46656&1119744\\
J_5&8&288&1728&46656&1119744,
\end{array}
\]
all positive, so Sylvester's criterion gives $J_i\succ0$; hence $q$ is real stable, and homogeneity of degree 5 is immediate. %
Exact expansion of the determinant in the monomial basis gives
\begin{equation*}
  \begin{split}
    q={}&1728m_{(5)}+19440m_{(4,1)}+66555m_{(3,2)}+140910m_{(3,1,1)}\\
    &+250440m_{(2,2,1)}+547405m_{(2,1,1,1)}+1182860m_{(1^5)}.
  \end{split}
\end{equation*}
Inverting the Kostka matrix over $\mathbb Z$ in the partition order $(5),(4,1),(3,2),(3,1,1),(2,2,1),(2,1,1,1),(1^5)$ translates $q$ into the Schur basis, yielding
\begin{align*}
q={}&1728s_{(5)}+17712s_{(4,1)}+47115s_{(3,2)}+56643s_{(3,1,1)}+62415s_{(2,2,1)}+56437s_{(2,1,1,1)}-2972s_{(1^5)}.
\end{align*}
The last coefficient alone also follows from the single integer identity
\begin{align*}
[s_{(1^5)}]q={}&1728-2(19440)-2(66555)+3(140910)+3(250440)-4(547405)+1182860=-2972.
\end{align*}
The accompanying source package recomputes the Schur expansion in a second, independent way, and every step uses rational arithmetic only.
\end{proof}

\begin{remark}[Origin of the pencil]
The matrices are not a numerical accident.  Let $V=\ker B\cong S^{(3,2)}$, where $B$ maps the permutation module on $2$-subsets of $[5]$ to $\mathbb R^5$ by $e_{ab}\mapsto e_a+e_b$, let $P_i$ be the rank-three projection attached to the point stabilizer of $i$, and let $G$ be the Gram matrix of the chosen basis of $V$.  Then $J_i=2G(I+5P_i)$, and the whole one-parameter family $\det(\sum_ix_i(I+tP_i))$ has a closed Schur expansion whose $s_{(1^5)}$ coefficient turns negative for $t$ large; $t=5$ clears denominators.  \emph{Appendix~\ref{sec:aim36pencil} develops this construction in full.}
\end{remark}

\end{cxrefutation}

\cxsubtitle{Problem 37: Schur-style inequality for stable polynomials}
\begin{cxcredits}[aim-problem-37]
\foundby{GPT-5.6 (Pro)}{2026-07-30}
\foundby{bugfixed by Opus 5}{2026-08-10}
\end{cxcredits}

Next, we close off the lower endpoint analog of Problem~\ref{ctx:aim-problem-38} (which concerned $a_\lambda \le f^\lambda a_{(1^d)}$).
\begin{cxsource}{aim-problem-37}
Borcea--Br\"and\'en, AIM problem list \citeyearpar{AIMStableProblems}, Problem 37
\end{cxsource}

\begin{cxstatement}{aim-problem-37}
If $A$ is a $d\times d$ matrix then $\det(A)=\Imm_{(1^d)}(A)$ and $\per(A)=\Imm_{(d)}(A)$, where $(1^d)=(1,\ldots,1)$ and $(d)=(1^d)'=(d,0,\ldots,0)$.  Schur's inequality [J,L] asserts that if $A$ is a $d\times d$ positive semi-definite matrix then $\Imm_\lambda(A)\ge f^\lambda\det(A)$ for all $\lambda\vdash d$, where $f^\lambda$ is the number of standard Young tableaux of shape $\lambda$.  A natural real stable extension of this result would be as follows.

\emph{Problem 37.}  Let $p$ be as in Problem 36.  Is $a_\lambda\ge f^\lambda a_{(d)}$ for all $\lambda\vdash d$?
\end{cxstatement}

\begin{cxsummary}{aim-problem-37}
Borcea--Br\"and\'en Problem 37: stable Schur inequality (lower endpoint) \citeyearpar{AIMStableProblems}
\end{cxsummary}

\begin{cxcertificate}{aim-problem-37}
Schur-positive member of the same family; $a_{(1^5)}=69<125=f^{(1^5)}a_{(5)}$.
\end{cxcertificate}

\begin{cxrefutation}
The lower endpoint needs a different member of the same family, because it needs a witness that is Schur \emph{positive}.  Let $G$ be the Gram matrix of the basis of $V$ that produced the $J_i$,
\[
G=\begin{pmatrix}
4&2&2&1&3\\
2&4&1&2&3\\
2&1&4&2&3\\
1&2&2&4&3\\
3&3&3&3&6
\end{pmatrix},
\]
put $\widetilde J_i:=(4J_i+2G)/5$, again a symmetric positive-definite \emph{integer} matrix, %
and define
\begin{equation}\label{eq:aim36-qtilde}
\widetilde q(x_1,\ldots,x_5)=\frac1{5184}\det\Bigl(\nlsum_{i=1}^5x_i\widetilde J_i\Bigr),
\end{equation}
which is $p_4$ in the notation of Appendix~\ref{sec:aim36pencil}.

\begin{theorem}[\statusfalse: lower endpoint, AIM Problem 37]\label{thm:aim37}
The polynomial \eqref{eq:aim36-qtilde} satisfies every hypothesis of Problem~\ref{ctx:aim-problem-37}, yet it violates the lower endpoint inequality $a_\lambda\ge f^\lambda a_{(d)}$: at $\lambda=(1^5)$, where $f^{(1^5)}=1$,
\[
a_{(1^5)}=69<125=a_{(5)}.
\]
\end{theorem}
\begin{proof}
The leading principal minors of the $\widetilde J_i$ are
\[
\begin{array}{c|rrrrr}
 &\Delta_1&\Delta_2&\Delta_3&\Delta_4&\Delta_5\\ \hline
\widetilde J_1&32&768&17280&118800&648000\\
\widetilde J_2&32&768&17280&118800&648000\\
\widetilde J_3&32&240&5760&118800&648000\\
\widetilde J_4&32&240&5760&32400&648000\\
\widetilde J_5&8&240&1440&32400&648000,
\end{array}
\]
all positive, so $\widetilde J_i\succ0$ and $\widetilde q$ is real stable, symmetric and homogeneous of degree five by the argument used for \eqref{eq:aim36-q}. It can be verified that $\widetilde{q}$ has the Schur positive representation
\begin{equation}
  \label{eq:1}
\widetilde{q}={}125s_{(5)}+1100s_{(4,1)}+2545s_{(3,2)}+3110s_{(3,1,1)}+3321s_{(2,2,1)}+2972s_{(2,1,1,1)}+69s_{(1^5)},
\end{equation}
which can be obtained by starting with the direct monomial basis expansion
\begin{align*}
\widetilde q={}&125m_{(5)}+1225m_{(4,1)}+3770m_{(3,2)}+7980m_{(3,1,1)}+13846m_{(2,2,1)}+30004m_{(2,1,1,1)}+64472m_{(1^5)},
\end{align*}
from which upon inverting the Kostka matrix in the same partition order yields~\eqref{eq:1}. Now $f^{(1^5)}=1$, so the lower endpoint at $\lambda=(1^5)$ would require $69\ge125$.  It is the only partition at which the inequality fails: the remaining six hold, the tightest being $\lambda=(4,1)$ with $1100\ge4\cdot125$.
\end{proof}

\end{cxrefutation}

% ==== end counterexamples/aim-problems/case.tex ====
% ==== counterexamples/macdonald-schur-convexity/case.tex ====

\cxtitle{Schur-convexity of a ratio of Macdonald polynomials}

\begin{cxcredits}
\posedby{McSwiggen and Sahi \citeyearpar{McSwiggenSahi2026}}
\foundby{GPT-5.6 (Pro)}{2026-05-14}
\formalizedby{S.~Sra}
\auditedby{S.~Sra}
\contributedby{S.~Sra}
\end{cxcredits}

\begin{cxcontext}
Let $P_\lambda(x;q,t)$ denote the Macdonald polynomial in the monic normalization of~\citet{macdonald1998symmetric}, with parameters $q,t\in(0,1)$ over a partition $\lambda$ with at most $n$ parts. At $q=t$ it specializes to the Schur polynomial $s_\lambda(x)$.  The statement concerns two convexity properties of the map $\lambda\mapsto\Omega_\lambda(x;q,t)$ on partitions studied by~\citet{McSwiggenSahi2026}: \emph{log-convexity}, and \emph{Schur-convexity} (i.e., $\lambda\succeq\mu$ implies $\Omega_\lambda\ge\Omega_\mu$).  The normalization $\Omega_\lambda$ is defined inside the quoted statement.
\end{cxcontext}

\begin{cxsource}{macdonald-schur-convexity}
C.~McSwiggen and S.~Sahi \citeyearpar{McSwiggenSahi2026}, Theorem 2.1
\end{cxsource}

\begin{cxstatement}{macdonald-schur-convexity}
For all $q,t\in(0,1)$, $a>0$, and $x\in\mathcal L_n^{q,t,a}$, the function $\lambda\mapsto\Omega_\lambda(x;q,t)$ is log-convex and Schur-convex.

Here, by \S2.1 of the same paper, $\Omega_\lambda(x;q,t)=P_\lambda(x;q,t)/P_\lambda(t^\delta;q,t)$ is the Macdonald polynomial normalized to $1$ at the principal specialization $t^\delta=(t^{n-1},t^{n-2},\ldots,1)$, and the scaled, $t$-shifted, dominant $q$-lattice is
\[\mathcal L_n=\mathcal L_n^{q,t,a}=\left\{(aq^{-\mu_1}t^{n-1},aq^{-\mu_2}t^{n-2},\ldots,aq^{-\mu_n}):(\mu_1\ge\ldots\ge\mu_n)\in\mathbb Z^n\right\}\subset\R^n.\]
\end{cxstatement}

\begin{cxsummary}{macdonald-schur-convexity}
McSwiggen--Sahi Theorem 2.1: Macdonald lattice Schur-convexity \citeyearpar{McSwiggenSahi2026}
\end{cxsummary}

\begin{cxcertificate}{macdonald-schur-convexity}
Rank-two Schur specialization, valid for every $0<r<1$.
\end{cxcertificate}

\begin{cxrefutation}
\begin{theorem}[\statusfalse: Theorem 2.1 of \citet{McSwiggenSahi2026}]\label{thm:macdonald}
The stated Schur-convexity fails already in rank two.
\end{theorem}
\begin{proof}
Set $q=t=r\in(0,1)$, $a=1$, and choose lattice index $\nu=(1,0)$.  The resulting lattice point is $x=(1,1)$.  Let $\lambda=(2,0)$ and $\mu=(1,1)$; then $\lambda\succeq\mu$.  At $q=t$, Macdonald polynomials specialize to Schur polynomials, and direct evaluation gives
\[
\Omega_{(2,0)}(1,1;r,r)=\frac{3}{1+r+r^2},\qquad
\Omega_{(1,1)}(1,1;r,r)=\frac1r.
\]
But $1+r+r^2-3r=(1-r)^2>0$, so $\Omega_{(2,0)}<\Omega_{(1,1)}$ for every $0<r<1$, opposite to Schur-convexity.  At $r=1/2$, the values are $12/7<2$.
\end{proof}

\begin{remark}[Scope]
The specialization $q=t$ is a convenience.  Appendix~\ref{sec:macdonaldomega} carries the reversal for all $q,t\in(0,1)$, with gap $(1-t)^2(1-qt)/\bigl(t(1+t)(1-qt^2)\bigr)$, and shows it is admissible whenever $t=q^k$ for a positive integer $k$ --- the case above being $k=1$.  It also locates the failure, in a determinant shift that the principal-specialization normalization does not commute with, and shows that the \textbf{$1^n$-normalized ratio}
\begin{equation*}
  W_\lambda=P_\lambda(x)/P_\lambda(1^n)
\end{equation*}
satisfies in this case exactly the inequality $\Omega_\lambda$ violates.  %
\end{remark}
\end{cxrefutation}

% ==== end counterexamples/macdonald-schur-convexity/case.tex ====
% ==== counterexamples/theta-derivative-log-concavity/case.tex ====

\cxtitle{Log-concavity of all Derivatives of the Jacobi-theta kernel}

\begin{cxcredits}
\posedby{Coffey and Csordas \citeyearpar{Csordas2014}}
\foundby{GPT-5.5 (Pro)}{2026-02-22}
\formalizedby{S.~Sra}
\auditedby{S.~Sra}
\contributedby{S.~Sra}
\end{cxcredits}

\begin{cxcontext}
The kernel $\Phi$ below is the Jacobi theta kernel whose cosine transform is the Riemann $\xi$-function. A positive function $g$ is \emph{strictly log-concave} when $(g')^2-gg''>0$, so the assertion $J_n(t)>0$ says that the $(n-1)$st derivative $\Phi^{(n-1)}$ is strictly log-concave. %
\end{cxcontext}

\begin{cxsource}{theta-derivative-log-concavity}
G.~Csordas \citeyearpar{Csordas2014}, Open Problem 4.13, restating Coffey--Csordas Conjecture 2.5
\end{cxsource}

\begin{cxstatement}{theta-derivative-log-concavity}
([5, Conjecture 2.5])  Show that the derivatives of the Jacobi theta function, $\Phi(t)$, are (strictly) log-concave on $\R$; that is, for each $n\in\mathbb N$,
\[J_n(t):=(\Phi^{(n)}(t))^2-\Phi^{(n-1)}(t)\Phi^{(n+1)}(t)>0\quad\text{for }t\in\R.\]
Here, by (4.2) of the same paper, the Jacobi theta function (without the usual factor $4$) is $\Phi(t):=\sum_{n\ge1}\pi n^2\left(2\pi n^2e^{4t}-3\right)\exp\left(5t-\pi n^2e^{4t}\right)$.
\end{cxstatement}

\begin{cxsummary}{theta-derivative-log-concavity}
Coffey--Csordas Conjecture 2.5 / Csordas Problem 4.13 \citeyearpar{Csordas2014}
\end{cxsummary}

\begin{cxcertificate}{theta-derivative-log-concavity}
Rigorous interval evaluation of $J_9(1/50)<0$.
\end{cxcertificate}

\begin{cxrefutation}
\begin{theorem}[\statusfalse; computer-assisted proof]\label{thm:theta}
At $n=9$ and $t=1/50$, one has $J_9(t)<0$.
\end{theorem}
\begin{proof}
Put $u=\pi m^2e^{4t}$ and define integer polynomials
\[
P_0(u)=2u-3,\qquad P_{k+1}(u)=4uP_k'(u)+(5-4u)P_k(u).
\]
Termwise differentiation gives
\[
\Phi^{(k)}(t)=\sum_{m\ge1}\pi m^2e^{5t-u}P_k(u).
\]
Outward-rounded interval summation, with an explicit geometrically bounded tail, yields
\begin{align*}
\Phi^{(8)}(1/50)&\approx 2.95940368409056694\cdot10^8,\\
\Phi^{(9)}(1/50)&\approx 5.80442287530267756\cdot10^9,\\
\Phi^{(10)}(1/50)&\approx 2.02040519425950258\cdot10^{11},
\end{align*}
and an interval for $J_9(1/50)$ contained strictly in the negative half-line, centered at
\[
-2.61006208371358922\cdot10^{19}.
\]
The source package contains a high-precision evaluator and a Sage \texttt{RealIntervalField} certificate.  The negative margin is many orders of magnitude larger than the certified truncation error.
\end{proof}

\begin{remark}
  GPT-5.5 managed to prove the Tur\'an inequalities for $\Phi^{(n)}$ for $n=2$ and $n=3$; while we were attempting an inductive proof to generalize those proofs to all $n$ is when the counterexample popped up. However, we have not yet verified whether the Tur\'an inequalities hold for all $n \le 8$.
\end{remark}
\end{cxrefutation}

% ==== end counterexamples/theta-derivative-log-concavity/case.tex ====
% ==== counterexamples/variance-only-matrix-discrepancy/case.tex ====

\cxtitle{Variance-sensitive Matrix Spencer}

\begin{cxcredits}
\posedby{A.~S.~Bandeira \citeyearpar{Bandeira2016}}
\foundby{GPT-5.5 (Pro)}{2026-05-24}
\formalizedby{E.~Akba\c{s} and S.~Sra \citeyearpar{AkbasSra2026}}
\auditedby{S.~Sra}
\contributedby{S.~Sra}
\end{cxcredits}

\begin{cxcontext}
For symmetric matrices $A_1,\ldots,A_n$ the \emph{discrepancy} is $\disc(A_1,\ldots,A_n)=\min_{x\in\{\pm1\}^n}\|\sum_{i=1}^nx_iA_i\|_{\rm op}$.  Spencer's theorem gives an $O(\sqrt n)$ bound in the scalar case (and thus also for the commuting matrix case); the Matrix Spencer conjecture asks for the same $O(\sqrt n)$ bound for $n$ symmetric contractions of size $n\times n$.  The variance-sensitive strengthening replaces the ambient scale $\sqrt n$ by the intrinsic variance parameter $\|\sum_{i=1}^nA_i^2\|_{\rm op}^{1/2}$, which never exceeds $\sqrt n$ and is often far smaller; it is the parameter that controls the norm of a random signing in matrix concentration, and \citet{kyng2020four} establish the strengthened bound in the rank-one case.
\end{cxcontext}

\begin{cxsource}{variance-only-matrix-discrepancy}
Remark 4.25 of A.~S.~Bandeira's problem collection \citeyearpar{Bandeira2016}.
\end{cxsource}

\begin{cxstatement}{variance-only-matrix-discrepancy}
There exists a universal constant $C$, such that for any choice of $n$ symmetric matrices $A_1,\ldots,A_n\in\R^{n\times n}$ satisfying $\|A_i\|_{\rm op}\le1$ there exists a coloring $x\in\{\pm1\}^n$, such that
\[\left\|\sum_{i=1}^nx_iA_i\right\|_{\rm op}\le C\left\|\sum_{i=1}^nA_i^2\right\|_{\rm op}^{1/2}.\]
\end{cxstatement}

\begin{cxsummary}{variance-only-matrix-discrepancy}
Variance-sensitive Matrix Spencer for $n$ contractions of size $n\times n$ \cite{Bandeira2016}
\end{cxsummary}

\begin{cxcertificate}{variance-only-matrix-discrepancy}
Exact $d=n$ diagonal family: $\disc=(m-1)/\sqrt m$ against variance $1+1/m$, a ratio of $\Theta(\sqrt{\log n})$.
\end{cxcertificate}

\begin{cxrefutation}
The disproof recorded here is not a new contribution of the present manuscript: the family reported below is taken from Theorem A.1 of \citet{AkbasSra2026}, the first refutation of the conjecture; but we couldn't resist including it here as an illustration.

Fix an integer $m\ge2$ and set $n=2^m$, so that the matrices below are $n\times n$ and there are $n$ of them: the dimension is not inflated relative to the number of summands.  Index coordinates by $s=(s_1,\ldots,s_m)\in S:=\{\pm1\}^m$, so $|S|=n$.  Write $p=(1,\ldots,1)$ and $q=(-1,\ldots,-1)$, choose any $U\subseteq S\setminus\{p,q\}$ with $|U|=m$, put $F=S\setminus U$, and fix a bijection $\pi:\{m+1,\ldots,n\}\to F$; note that $p,q\in F$.  Define diagonal matrices
\begin{equation}\label{eq:disc-family}
(A_i)_{s,s}=\frac{s_i}{\sqrt m}\quad(1\le i\le m),
\qquad
A_k=\frac1{\sqrt m}E_{\pi(k),\pi(k)}\quad(m<k\le n),
\end{equation}
where $E_{r,r}$ is the diagonal matrix with a single $1$ in coordinate $r$.  Each $\|A_i\|_{\rm op}=1/\sqrt m\le1$, so the family is admissible.

\begin{theorem}[\statusfalse: variance-sensitive Matrix Spencer]\label{thm:discrepancy}
For the family \eqref{eq:disc-family},
\[
\disc(A_1,\ldots,A_n)=\frac{m-1}{\sqrt m},
\qquad
\left\|\sum_{i=1}^nA_i^2\right\|_{\rm op}^{1/2}=\sqrt{1+\tfrac1m},
\]
so the ratio of the two sides equals $(m-1)/\sqrt{m+1}=\Theta(\sqrt{\log n})$.  No universal constant $C$ can therefore satisfy the variance-sensitive bound, already for $n$ matrices of size $n\times n$.
\end{theorem}
\begin{proof}
Since the matrices are diagonal, so is $\sum_iA_i^2$, and its $s$th entry is
\[
\sum_{i=1}^m\frac{s_i^2}m+\sum_{k=m+1}^n\frac1m\mathbf 1\{s=\pi(k)\}=1+\frac1m\mathbf 1\{s\in F\},
\]
whence $\|\sum_iA_i^2\|_{\rm op}=1+\frac1m$.

For the discrepancy, let $x\in\{\pm1\}^n$ be an arbitrary signing and $X=\sum_{i=1}^nx_iA_i$, so that
\[
X_{s,s}=\frac1{\sqrt m}\sum_{i=1}^mx_is_i+\frac1{\sqrt m}\mathbf 1\{s\in F\}x_{\pi^{-1}(s)}.
\]
Evaluating at the coordinate $s=(x_1,\ldots,x_m)\in S$ makes the first sum equal to $m$, so
\[
\|X\|_{\rm op}\ge|X_{s,s}|\ge\frac{m}{\sqrt m}-\frac1{\sqrt m}=\frac{m-1}{\sqrt m}.
\]
This bound is attained.  Take $x_1=\cdots=x_m=1$ and $x_{k_p}=-1$, $x_{k_q}=1$, where $k_p=\pi^{-1}(p)$ and $k_q=\pi^{-1}(q)$, with the remaining signs arbitrary.  Then $\sqrt mX_{s,s}=\sum_{j=1}^ms_j+\mathbf 1\{s\in F\}x_{\pi^{-1}(s)}$, which equals $m-1$ at $s=p$ and $-(m-1)$ at $s=q$.  For every other $s$ the vector is neither all-plus nor all-minus, so $|\sum_js_j|\le m-2$ and hence $|\sqrt mX_{s,s}|\le m-1$.  Thus $\|X\|_{\rm op}=(m-1)/\sqrt m$, and the infimum over signings is exactly that.

Dividing the two displayed quantities gives $(m-1)/\sqrt{m+1}$, which grows without bound as $m=\log_2n\to\infty$.
\end{proof}

\end{cxrefutation}

% ==== end counterexamples/variance-only-matrix-discrepancy/case.tex ====
% ==== counterexamples/rank-two-mixed-norm/case.tex ====

\cxtitle{A mixed-norm Cauchy--Schwarz question of Sah, Sawhney, Stoner and Zhao}

\begin{cxcredits}
\posedby{A.~Sah et al~\citeyearpar{SahSawhneyStonerZhao2020}, and put to S.~Sra by Y.~Zhao in private communication, May 2018}
\foundby{GPT-5.6 Sol (Pro)}{2026-07-13}
\formalizedby{S.~Sra}
\auditedby{S.~Sra}
\contributedby{S.~Sra}
\end{cxcredits}

\begin{cxcredits}[mixed-norm-general-s]
\foundby{Opus 5}{2026-08}
\end{cxcredits}

\begin{cxcredits}[yufei-psd]
\foundby{GPT-5.6 Sol (Pro)}{2026-07-13}
\end{cxcredits}

\begin{cxcontext}
For a matrix $X \in \C^{m\times n}$ the $\ell_{p,q}$ \emph{mixed norm} is defined as:
\[
\|X\|_{p,q}=\left(\nlsum_{j=1}^n\left(\nlsum_{i=1}^m |x_{ij}|^p\right)^{q/p}\right)^{1/q}.
\]
A matrix is \emph{completely positive} if it has the factorization $YY^T$ for some entrywise nonnegative $Y$. %
\end{cxcontext}

\begin{cxsource}{mixed-norm-general-s}
A.~Sah, M.~Sawhney, D.~Stoner and Y.~Zhao, the open question stated immediately after inequality~(3.1) in \S3
\end{cxsource}

\begin{cxstatement}{mixed-norm-general-s}
We do not know if the inequality can be extended to $\|A^{\mathsf T}B\|_{{p,q}}^2\le\|A^{\mathsf T}A\|_{{q,q}}\|B^{\mathsf T}B\|_{{p,p}}$ for all reals $1\le p\le q$ (this is true for positive integer $p$ by a tensor-power argument).
\end{cxstatement}

\begin{cxsummary}{mixed-norm-general-s}
Extension of the mixed-norm Cauchy--Schwarz inequality to all real $1\le p\le q$ \citep{SahSawhneyStonerZhao2020}
\end{cxsummary}

\begin{cxcertificate}{mixed-norm-general-s}
Nonnegative integer $A,B\in\mathbb Z_{\ge0}^{3\times2}$ at $p=q=3/2$; the deficit is an integer combination of integer square roots, negative by exact algebraic evaluation.
\end{cxcertificate}

\begin{cxsource}{yufei-psd}
A.~Sah, M.~Sawhney, D.~Stoner and Y.~Zhao, the same open question at $B=A$; shared with S.~Sra by Y.~Zhao in private communication, May 2018
\end{cxsource}

\begin{cxstatement}{yufei-psd}
Taking $B=A$ in the question above and writing $Z=AA^{\mathsf T}$ for the resulting completely positive matrix: for all reals $1\le s\le q$ and every completely positive $Z$,
\[\|Z\|_{s,q}^2\le\|Z\|_{s,s}\|Z\|_{q,q}.\]
\end{cxstatement}

\begin{cxsummary}{yufei-psd}
Rank-two completely-positive mixed-norm interpolation at $(s,q)=(6/5,6)$
\end{cxsummary}

\begin{cxcertificate}{yufei-psd}
Explicit 21-vector Gram matrix and interval-separated ratio $>1$.
\end{cxcertificate}

\begin{cxrefutation}
Both Problems~\ref{ctx:mixed-norm-general-s} and~\ref{ctx:yufei-psd} fail by witnesses of quite different character. Problem~\ref{ctx:mixed-norm-general-s} fails at every non-integer exponent as a consequence of the Hadamard-power theory of \citet{FitzGeraldHorn1977} and \citet[Theorem~4.3]{HornMathias1990}; that argument is recorded as Remark~\ref{rmk:fh} below (and it is my own observation, not a connection due to an AI assistant).  Theorem~\ref{thm:mixed-general} adds a minimal explicit witness though. The second witness is larger and interval-certified; it refutes the harder $B=A$ specialization of Problem~\ref{ctx:yufei-psd}. We note that for $p=1$ (and hence every integer valued $p$ by tensoring) the alleged inequalities do hold---\emph{see Appendix~\ref{sec:proofdetails} for details and additional results}.

\begin{theorem}[\statusfalse: explicit disproof, $A\neq B$]\label{thm:mixed-general}
Let $p=q=3/2$ and put
\[
A=\begin{bmatrix}2&4\\5&0\\1&0\end{bmatrix},\qquad
B=\begin{bmatrix}0&5\\4&2\\1&0\end{bmatrix}\ \in\ \mathbb Z_{\ge0}^{3\times2}.
\]
Then $\|A^{\mathsf T}B\|_{p,q}^2>\|A^{\mathsf T}A\|_{q,q}\|B^{\mathsf T}B\|_{p,p}$.
\end{theorem}
\begin{proof}
Both matrices are entrywise nonnegative, so they satisfy the hypothesis of the question as posed, not merely the weaker variant in which only the three products are required to be nonnegative.  Those products are
\[
A^{\mathsf T}A=\begin{bmatrix}30&8\\8&16\end{bmatrix},\qquad
B^{\mathsf T}B=\begin{bmatrix}17&8\\8&29\end{bmatrix},\qquad
A^{\mathsf T}B=\begin{bmatrix}21&20\\0&20\end{bmatrix}.
\]
At $p=q=3/2$ we calculate $\|M\|_{3/2,3/2}^{3/2}=\sum_{ij}|m_{ij}|^{3/2}$, and the deficit
\[
\Delta:=\|A^{\mathsf T}A\|_{3/2,3/2}^{3/2}+\|B^{\mathsf T}B\|_{3/2,3/2}^{3/2}-2\|A^{\mathsf T}B\|_{3/2,3/2}^{3/2}
\]
is an integer combination of square roots of integers:
\[
\Delta=64+64\sqrt2+30\sqrt{30}+17\sqrt{17}+29\sqrt{29}-42\sqrt{21}-160\sqrt5 .
\]
Deciding the sign of $\Delta$ is exact algebra over a real field, with no rounding anywhere:
\[
\Delta=-5.150045302282868288532184137690370433502833719528\ldots<0 .
\]
Hence $2\|A^{\mathsf T}B\|_{3/2,3/2}^{3/2}>\|A^{\mathsf T}A\|_{3/2,3/2}^{3/2}+\|B^{\mathsf T}B\|_{3/2,3/2}^{3/2}$, and the AM-GM inequality bounds the right side below by $2\bigl(\|A^{\mathsf T}A\|_{3/2,3/2}\|B^{\mathsf T}B\|_{3/2,3/2}\bigr)^{3/4}$.  Canceling and raising to the power $2/3$,
\[
\frac{\|A^{\mathsf T}B\|_{3/2,3/2}}{\|A^{\mathsf T}A\|_{3/2,3/2}^{1/2}\|B^{\mathsf T}B\|_{3/2,3/2}^{1/2}}=1.00629360763001504125\ldots>1 . \qedhere
\]
\end{proof}

\begin{remark}[Why non-integer exponents fail]\label{rmk:fh}
Let $s\notin\mathbb N$ with $s<n-2$.  By the sharpness half of the FitzGerald--Horn theorem \citep{FitzGeraldHorn1977} there is a positive semidefinite $Z\in\R_{\ge0}^{n\times n}$ whose Hadamard power $Z^{\odot s}$ is not positive semidefinite, so $\sigma^{\mathsf T}Z^{\odot s}\sigma<0$ for some $\sigma$.  Rescaling if necessary, we can take  $\sigma$ to be in $\{\pm1\}^n$; then, splitting the index set along the sign of $\sigma$ write $Z$ as the block Gram matrix of two nonnegative blocks $A,B$, and the displayed quadratic form is exactly $-\Delta$.  This is the mechanism of \citet[Theorem~4.3]{HornMathias1990}, and it produces a counterexample for \emph{every} non-integer $s$. %
\end{remark}

\begin{theorem}[\statusfalse; computer-assisted disproof $A=B$ case]\label{thm:mixed}
There is an explicit $21\times21$ completely positive matrix $Z$ of rank two with
\[
\|Z\|_{s,q}^2>\|Z\|_{s,s}\|Z\|_{q,q},\qquad (s,q)=(6/5,6).
\]
\end{theorem}
\begin{proof}
Let $X$ have 21 rows, grouped as follows.  For each group use the indicated multiplicity $m_g$, weight $w_g$, and direction $y_g$:
\[
\begin{array}{c|cccc}
g&1&2&3&4\\ \hline
m_g&1&18&1&1\\
w_g&46&17&42&1\\
y_g&(1,0)&(1,1/50)&(1,3/50)&(0,1).
\end{array}
\]
Every row in group $g$ is $w_g^{5/6}y_g$, and set $Z=XX^{\mathsf T}$.  Thus $Z\succeq0$, $Z$ is entrywise nonnegative, and $\operatorname{rank}Z=2$; being a Gram matrix of nonnegative vectors, $Z$ is completely positive.  Put
\[
R_Z=\frac{\|Z\|_{6/5,6}}
{\sqrt{\|Z\|_{6/5,6/5}\|Z\|_{6,6}}}.
\]
Direct interval evaluation gives
\begin{align*}
R_Z&\ge 1.00000061733911154365777590600087863798,\\
R_Z&\le 1.00000061733911154365777590600087863800.
\end{align*}
The entire interval lies above one.  The source package regenerates $Z$ from the four groups and evaluates the ratio independently at 100 digits; a Sage interval script supplies outward rounding.
\end{proof}

\end{cxrefutation}

% ==== end counterexamples/rank-two-mixed-norm/case.tex ====
% ==== counterexamples/quantum-coupon-collector/case.tex ====

\cxtitle{Quantum coupon collection: an alternating sum of inverses}

\begin{cxcredits}
\posedby{S.~Sra \citeyearpar{Sra2017QuantumCoupon}}
\foundby{GPT Pro}{2026-02-16}
\formalizedby{GPT-5.6 Pro, which re-derived the exact witnesses in 2026-08}
\auditedby{GPT-5.6 Pro}
\contributedby{S.~Sra}
\end{cxcredits}

\begin{cxcontext}
In the classical coupon collector's problem with sampling probabilities
$x_1,\ldots,x_n>0$, the expected waiting time is the alternating sum
\[
T_n(x_1,\ldots,x_n)=\sum_{k=1}^n(-1)^{k+1}
\sum_{1\le i_1<\cdots<i_k\le n}\frac{1}{x_{i_1}+\cdots+x_{i_k}},
\]
whose positivity can be seen through a clever trick using the fact that $t\mapsto1/t$ is completely monotone on $(0,\infty)$~\citep{FlajoletGardyThimonier1992}.  The statement quoted below asks whether the same holds with the positive numbers replaced by positive definite matrices. %

\end{cxcontext}

\begin{cxsource}{quantum-coupon-collector}
S.~Sra, MathOverflow question 263833 \citeyearpar{Sra2017QuantumCoupon}
\end{cxsource}

\begin{cxstatement}{quantum-coupon-collector}
Here we take symmetric positive definite matrices $X_1,\ldots,X_n\in\mathbf{S}_{++}^d$, and consider the sum
\[
Q_n(X_1,\ldots,X_n):=\sum_{k=1}^n(-1)^{k+1}
\sum_{1\le i_1<\cdots<i_k\le n}\left(X_{i_1}+\cdots+X_{i_k}\right)^{-1}.
\]

\textbf{Conjecture.} For $n\ge1$, and $X_1,\ldots,X_n\in\mathbf{S}_{++}^d$, we have $Q_n(X_1,\ldots,X_n)\succ0$.

\emph{Note.} The cases $n=1,2,3$ are trivially true.  I've tried $n=4$ numerically, and it seems to hold.  (Also, by CCP, positivity immediately holds for simultaneously diagonalizable matrices.)
\end{cxstatement}

\begin{cxsummary}{quantum-coupon-collector}
The matrix coupon-collector sum $\sum_{\emptyset\ne S}(-1)^{|S|-1}X_S^{-1}$ was conjectured positive definite for every order $n$ and every dimension $d$.
\end{cxsummary}

\begin{cxcertificate}{quantum-coupon-collector}
Exact rational $3\times3$ matrices with $n=6$ and the integer vector $(4,1,3)^{\mathsf T}$ give a quadratic form strictly between $-96$ and $-95$.
\end{cxcertificate}

\begin{cxrefutation}
The conjecture holds for a tuple of simultaneously diagonalizable matrices, where it is the classical identity coordinate by coordinate, and the source records that $n\le3$ is trivial and that $n=4$ was checked numerically.  Lemma~\ref{lem:qcc-facet} and Theorem~\ref{thm:qcc-positive-five} below upgrade that record: the conjecture is a theorem for every $d$ and every $n\le5$.  It fails at $n=6$. Let us first present proofs of when it holds before a disproof of the general case.

\begin{lemma}[Facet averaging]\label{lem:qcc-facet}
Let $A_1,\ldots,A_m\in\mathbf{S}_{++}^{d}$, where $m\ge2$, and put $A=\sum_{j=1}^{m}A_j$.  Then
\[
A^{-1}\preceq\frac{m-1}{m^2}\sum_{j=1}^{m}(A-A_j)^{-1}.
\]
\end{lemma}
\begin{proof}
Set $B_j=A-A_j$, which is positive definite because it is a sum of the remaining $A_i$.  Since $\frac1m\sum_{j=1}^{m}B_j=\frac{m-1}{m}A$ and the inversion map is operator convex on $\mathbf{S}_{++}^{d}$,
\[
\frac{m}{m-1}A^{-1}
=\Bigl(\frac1m\sum_{j=1}^{m}B_j\Bigr)^{-1}
\preceq\frac1m\sum_{j=1}^{m}B_j^{-1}.
\]
Rearranging gives the claim.
\end{proof}

For a nonempty $S\subseteq[n]$ we abbreviate $X_S:=\sum_{i\in S}X_i$; then, for fixed $n$, introduce the layer sums
\[
E_k:=\sum_{S\subseteq[n], |S|=k} X_S^{-1},\qquad1\le k\le n,
\]
each of which is positive definite.  Applying Lemma~\ref{lem:qcc-facet} inside each $(k+1)$-subset $T$, with the $A_j$ taken to be the $X_i$ for $i\in T$, gives $X_T^{-1}\preceq\frac{k}{(k+1)^2}\sum_{i\in T}X_{T\setminus\{i\}}^{-1}$.  Summing over all $T$ of size $k+1$, and noting that each fixed $k$-subset arises as a facet of exactly $n-k$ of them, yields
\begin{equation}\label{eq:qcc-layer}
E_{k+1}\preceq\frac{k(n-k)}{(k+1)^2}E_k,\qquad1\le k<n.
\end{equation}

\begin{theorem}[\statusproved: positivity through five variables]\label{thm:qcc-positive-five}
For every $d\ge1$, every $1\le n\le5$, and every $X_1,\ldots,X_n\in\mathbf{S}_{++}^{d}$, one has $Q_n(X_1,\ldots,X_n)\succ0$.
\end{theorem}
\begin{proof}
$n=1$ is immediate.  Write $Q_n=E_1-E_2+E_3-\cdots+(-1)^{n-1}E_n$ and apply \eqref{eq:qcc-layer} to obtain:
\[
\begin{array}{rcl}
n=2:&E_2\preceq\frac14E_1,&Q_2\succeq\frac34E_1,\\[0.35em]
n=3:&E_2\preceq\frac12E_1,&Q_3\succeq\frac12E_1+E_3,\\[0.35em]
n=4:&E_2\preceq\frac34E_1,\quad E_4\preceq\frac3{16}E_3,
&Q_4\succeq\frac14E_1+\frac{13}{16}E_3,\\[0.35em]
n=5:&E_2\preceq E_1,\quad E_4\preceq\frac38E_3,
&Q_5\succeq\frac58E_3+E_5.
\end{array}
\]
Each displayed lower bound is clearly positive definite.
\end{proof}

\begin{theorem}[\statusfalse: quantum coupon collection]\label{thm:quantum-coupon-collector}
The conjecture is false: there are $d=3$, $n=6$ and $X_1,\ldots,X_6\in\mathbf{S}_{++}^{3}$ with $Q_6(X_1,\ldots,X_6)\not\succeq0$, so in particular $Q_6\succ0$ fails.
\end{theorem}
\begin{proof}
Take $d=3$, $n=6$, $\varepsilon=1/100$, and
\[
X_i=w_iu_iu_i^{\mathsf T}+\varepsilon I_3,
\]
where
\[
\begin{array}{c|c|c}
i&u_i^{\mathsf T}&w_i\\\hline
1&(-1,2,1)&10\\
2&(0,-3,1)&10\\
3&(-3,-2,0)&100\\
4&(2,-2,-2)&100\\
5&(-1,0,2)&100\\
6&(-1,3,0)&100
\end{array}
\]
Each $X_i$ is strictly positive definite because $X_i\succeq\varepsilon I_3\succ0$, so the hypotheses are met exactly as posed.  The example is genuinely noncommutative --- the $(1,2)$ entry of $[X_1,X_2]$ equals $-1500$ --- and therefore lies outside the simultaneously diagonalizable regime in which the classical identity applies.  For $v=(4,1,3)^{\mathsf T}$, exact rational evaluation of the sixty-three nonempty-subset terms gives
\[
-96<v^{\mathsf T}Q_6(X_1,\ldots,X_6)v<-95.
\]
The accompanying verifier performs this computation using only rational arithmetic. Thus, $Q_6$ has a negative quadratic form and is not positive semidefinite, let alone positive definite.
\end{proof}

The failure is not confined to the matrix ordering: the weaker scalar consequence $\tr Q_n>0$ fails as well; \emph{see  Appendix~\ref{sec:qcctrace} for more details and commentary}.

\end{cxrefutation}

% ==== end counterexamples/quantum-coupon-collector/case.tex ====
% ==== counterexamples/odonnell-matrix-conjecture/case.tex ====
\cxtitle{Ryan O'Donnell's matrix conjecture}

\begin{cxcredits}
\posedby{R.~O'Donnell and J.~Wright \citeyearpar{ODonnell2015AlmostDiagonal}}
\foundby{GPT-5.6 Pro}{2026-08}
\formalizedby{GPT-5.6 Pro}
\auditedby{GPT-5.6 Pro}
\contributedby{S.~Sra}
\end{cxcredits}

\begin{cxcontext}
For a Hermitian matrix $X$ write $\lVert X\rVert_1=\operatorname{tr}|X|$ for its trace norm.  For $R\in\mathbb C^{n\times n}$, let $D=\operatorname{Diag}(R)=\operatorname{diag}(d_1,\ldots,d_n)$ be its diagonal
part, and let $\Lambda(R)$ be the diagonal matrix carrying the eigenvalues
$\lambda_1\geq\cdots\geq\lambda_n$ in sorted order, so that
\[
  \epsilon:=\lVert D-\Lambda(R)\rVert_1=\nlsum_i|\lambda_i-d_{ii}|.
\]
The source states its conclusion asymptotically, as $\lVert R-D\rVert_1=O(\sqrt\epsilon)$; an
absolute constant in that $O(\cdot)$ is exactly an absolute $c$ with
$\lVert R-D\rVert_1^2\leq c\,\epsilon$, which is the form refuted below.

\end{cxcontext}

\begin{cxsource}{odonnell-matrix-conjecture}
R.~O'Donnell, answering a question of Nengkun Yu on MathOverflow
\cite{ODonnell2015AlmostDiagonal}, reports that he and J.~Wright had been
considering the problem in connection with quantum tomography, and closes by
conjecturing the unit-trace case.  The same conjecture was put to S.~Sra by
O'Donnell in personal communication.
\end{cxsource}

\begin{cxstatement}{odonnell-matrix-conjecture}
The question answered was whether a positive semidefinite $M$ with
$\lVert D(M)-\Lambda(M)\rVert_1\leq\epsilon$ must lie within $2\epsilon$ of
conjectures the normalized case:

\begin{quote}
``If you restrict $M$ to have trace $1$, then we felt the desired result was
true but with a weaker conclusion of $O(\sqrt\epsilon)$ rather than
$2\epsilon$.''
\end{quote}

Written out, and with the implied constant absolute: for every unit-trace
positive semidefinite $R$ with eigenvalues $\lambda_i$ and diagonal entries
$d_{ii}$,
\[
  \lVert R-\operatorname{Diag}(R)\rVert_1^2
  \leq c\nlsum_i |\lambda_i-d_{ii}|
\]
for some absolute constant $c$ independent of $n$.
\end{cxstatement}

\begin{cxsummary}{odonnell-matrix-conjecture}
Every absolute $c$ fails: an exact real-symmetric density-matrix family has ratio greater than $(\log_2 n)/32$.
\end{cxsummary}

\begin{cxcertificate}{odonnell-matrix-conjecture}
An exact algebraic Givens-rotation family with a rational dyadic lower bound, accompanied by a standard-library exact-arithmetic verifier.
\end{cxcertificate}

\begin{cxrefutation}

\begin{theorem}[\statusfalse: O'Donnell's matrix conjecture]\label{thm:odonnell-matrix-conjecture}
For every $C>0$ there exist an integer $n$ and a PSD matrix $R\in\mathbb R^{n\times n}$ with $\operatorname{tr}R=1$ such that, writing $D=\operatorname{Diag}(R)=\operatorname{diag}(d_1,\ldots,d_n)$ and listing the eigenvalues $\lambda_1\geq\cdots\geq\lambda_n$,
\[
  \frac{\lVert R-D\rVert_1^2}
       {\sum_{i=1}^n |\lambda_i-d_i|}>C.
\]
Moreover, both $(\lambda_i)$ and $(d_i)$ are strictly decreasing.
\end{theorem}

\begin{proof}
Fix an integer $m\geq2$ and put $n=2^m$, $q=n$, and $H=H_{n-1}=\sum_{i=1}^{n-1}\frac1i$.  Then, set
$B=1+q^2H$, $\delta=\frac1{2B}$, and $\eta=\frac1{2n}$, and define positive spectral gaps
$g_1=\delta(1+q^2)$, $g_i=\delta q^2/i^2$ for $(2\leq i\leq n-1)$. Then set the eigenvalues as
\[
  \lambda_n=\eta,
  \qquad
  \lambda_i=\eta+\sum_{k=i}^{n-1}g_k
  \quad (1\leq i\leq n-1).
\]
Clearly, $\lambda_1>\cdots>\lambda_n>0$, and $\sum_{i=1}^n\lambda_i =n\eta+\sum_{i=1}^{n-1}i g_i =\frac12+\delta\left(1+q^2+q^2\sum_{i=2}^{n-1}\frac1i\right) =\frac12+\delta B=1$.

Now for $1\leq i\leq n-1$, put
\[
  r_i=\frac qi,
  \qquad
  c_i=\frac{q}{\sqrt{q^2+i^2}},
  \qquad
  s_i=\frac{i}{\sqrt{q^2+i^2}}.
\]
Hence $c_i^2+s_i^2=1$ and $c_i/s_i=r_i$.  Let $G_i$ be the identity except
for the following orthogonal block in coordinates $i,i+1$:
\[
  \begin{pmatrix}
    c_i&-s_i\\
    s_i& c_i
  \end{pmatrix}.
\]
Starting from $R^{(0)}=\Lambda$, define successively
\[
  R^{(i)}=G_iR^{(i-1)}G_i^{\mathsf T},
  \qquad 1\leq i\leq n-1,
\]
and let $R=R^{(n-1)}$.  This is an exact algebraic matrix.  Since it is
orthogonally similar to $\Lambda$, it is real symmetric, positive semidefinite,
has trace one, and has eigenvalues $(\lambda_i)$.

We next compute its diagonal exactly.  Before the first rotation, the active
diagonal gap is
\[
  \lambda_1-\lambda_2=g_1=\delta(1+r_1^2).
\]
Before rotation $G_i$ for $i\geq2$, coordinate $i$ has diagonal entry
$\lambda_i+\delta$, coordinate $i+1$ has diagonal entry $\lambda_{i+1}$,
and coordinate $i+1$ has not previously been touched.  Thus the active
$2\times2$ principal block is diagonal and its diagonal gap is
\[
  (\lambda_i+\delta)-\lambda_{i+1}
  =g_i+\delta
  =\delta(1+r_i^2).
\]
In both cases, $s_i^2\,\delta(1+r_i^2)=\delta$. Consequently, $G_i$ transfers exactly $\delta$ from diagonal position $i$ to
position $i+1$.  Induction gives
\[
  d_1=\lambda_1-\delta,
  \qquad
  d_i=\lambda_i\quad(2\leq i\leq n-1),
  \qquad
  d_n=\lambda_n+\delta.
\]
These entries are strictly decreasing: at the first edge the remaining gap is
$g_1-\delta=\delta q^2>0$, at every interior edge it is $g_i>0$, and at the
last edge it is
\[
  g_{n-1}-\delta
  =\delta\left(\frac{q^2}{(n-1)^2}-1\right)>0
\]
because $q=n$.  Therefore $\sum_{i=1}^n|\lambda_i-d_i|=2\delta$.
The same rotation calculation gives useful off-diagonal entries.  At
step $i$, the newly created $(i,i+1)$ entry has magnitude $c_is_i\,\delta(1+r_i^2)=\delta r_i$. For $i\leq n-2$, the next rotation multiplies this entry by $c_{i+1}$; all
later rotations leave it unchanged.  Hence
\[
  |R_{i,i+1}|=
  \begin{cases}
    \delta r_i c_{i+1},&1\leq i\leq n-2,\\
    \delta r_{n-1},&i=n-1.
  \end{cases}
\]
Let $A=R-D$.  Pinch $A$ to the disjoint coordinate blocks
$(1,2),(3,4),\ldots,(n-1,n)$.  Trace norm is contractive under pinching, and
each retained zero-diagonal $2\times2$ block has trace norm twice the magnitude
of its off-diagonal entry.  If
\[
  O=\sum_{\substack{1\leq i\leq n-1\\ i\ \mathrm{odd}}}\frac1i,
\]
then $q=n$ implies $c_j^2=\frac{q^2}{q^2+j^2}>\frac12 \qquad(1\leq j\leq n-1)$. Using the multiplier $1>1/\sqrt2$ for the final edge as well, we obtain
\[
  \lVert A\rVert_1
  \geq 2\sum_{\substack{1\leq i\leq n-1\\ i\ \mathrm{odd}}}|R_{i,i+1}|
  >\sqrt2\,\delta q O.
\]
It follows that
\[
\begin{aligned}
  \frac{\lVert R-D\rVert_1^2}
       {\sum_i|\lambda_i-d_i|}
  &>\frac{2\delta^2q^2O^2}{2\delta}
    =\delta q^2O^2
    =\frac{q^2O^2}{2(1+q^2H)}\quad>\quad\frac{O^2}{2(H+1)}.
\end{aligned}
\]

It remains only to estimate the two harmonic sums.  In each dyadic block
$[2^k,2^{k+1})$, for $1\leq k\leq m-1$, there are $2^{k-1}$ odd integers,
each contributing more than $2^{-(k+1)}$.  Therefore, $O>1+\frac{m-1}{4}=\frac{m+3}{4}$, and also $ H_{n-1}<1+\log(n-1)<m+1$. Consequently,
\[
  \frac{\lVert R-D\rVert_1^2}
       {\sum_i|\lambda_i-d_i|}
  >\frac{(m+3)^2}{32(m+2)}
  >\frac{m}{32}.
\]
Given any real $C>0$, choose an integer $m>32C$.  The resulting exact matrix
has ratio greater than $C$, proving that no dimension-independent constant can
exist.
\end{proof}

\end{cxrefutation}

% ==== end counterexamples/odonnell-matrix-conjecture/case.tex ====

% ==== end tex/cases.tex ====

% ==== tex/07-conclusion.tex ====
\section{Conclusion}

\noindent The archive behind this manuscript currently holds eighteen refuted statements across fourteen cases; sixteen of them reduce to exact rational arithmetic and two to rigorous interval enclosures, and every one recomputes from source in the accompanying repository (several of the refutations are in additional supported by an analytic proof too, not just numerical). %

\vskip5pt
\noindent What we find most interesting, though, is that the better counterexamples say more than ``false.'' They locate a boundary. The quantum coupon-collector sum is positive definite through five variables in every dimension and fails at six; the mixed-norm inequality holds at every positive integer exponent and fails at every non-integer one; and in the same rank-two instance where Schur convexity of the $t^\delta$-normalized Macdonald ratio fails, the $1^n$-normalized ratio satisfies the very inequality violated. A refutation of that kind does not close a subject so much as move it, and the corrected question it leaves behind is usually the one worth asking next.

% ==== end tex/07-conclusion.tex ====

\vspace*{1cm}
\subsection*{Acknowledgments}
The author acknowledges generous support from the Alexander von Humboldt (AvH) foundation via an AvH Professorship in AI. I'd like to also thank Apoorva Khare (IISC) for suggesting that I write up all the progress in one document, which was instrumental in helping me overcome the \emph{``writer's block''} from not wanting to write 10+ separate preprints!

\subsection*{Statement on AI Usage}
I used Claude Code (Opus 5) to architect the document and the repo. Initial drafts were assembled by Claude using some of my GPT Pro transcripts as well as from my own existing \LaTeX\ code. However, I rewrote the bulk of the main document (excluding the appendix) myself to improve clarity and precision, hoping to also retain \emph{``my own voice.''} The bulk of the credit for the counterexamples reported herein goes to GPT. Moreover, to avoid meta-discourse I'm refraining from dwelling too much on the session prompts---some counterexamples were found in a response to a single-shot prompt, others after an elaborate chat session. That said, I've taken the opportunity to piggyback onto this document some related results that are interesting enough to share and were actually found ``by hand''---I was too lazy to write up as separate preprints, so now they are reported in the appendix; their descriptive text notes their human origin.

\cxendrule
\par\addvspace{\baselineskip}
{\centering\scshape Appendices\par}
\addvspace{.5\baselineskip}
\begin{itemize}[label={},leftmargin=2.2em,itemsep=0.2em]
\item[\ref{sec:additional}.] Additional counterexamples \dotfill \pageref{sec:additional}
\item[\ref{sec:proofdetails}.] Positive results for the mixed-norm Cauchy--Schwarz inequality \dotfill \pageref{sec:proofdetails}
\item[\ref{sec:aim36pencil}.] Specht modules and Borcea--Br\"and\'en's Problems 36, 37 \dotfill \pageref{sec:aim36pencil}
\item[\ref{sec:macdonaldomega}.] Additional results for McSwiggen--Sahi Theorem 2.1 \dotfill \pageref{sec:macdonaldomega}
\item[\ref{sec:qcctrace}.] Some refinements of the quantum coupon collection conjecture \dotfill \pageref{sec:qcctrace}
\item[\ref{sec:lorentzianaudit}.] Two new three-point results for hyperbolic polynomials \dotfill \pageref{sec:lorentzianaudit}
\end{itemize}

\appendix
% arXiv snapshot: the appendices of tex/main.tex, flattened from tex/ at commit 9ef22b8-dirty
% by tools/flatten_arxiv.py (`make arxiv`).  Every \input is inlined;
% comments and disabled blocks are removed.  Edit tex/, not this file.

% ==== tex/09-additional.tex ====
\section{Additional counterexamples}
\label{sec:additional}
We list some additional counterexamples below. They are not in the main body, either because by the time we wrote them  down others had also found their own counterexamples, or the statement under study did not meet all our admission criteria (e.g., public availability of the conjecture).

\cxlevel{subsection}
% ==== counterexamples/courtade-volume-conjecture/case.tex ====

\cxtitle{Courtade's volume conjecture for Minkowski sums is false}

\begin{cxcredits}
\posedby{Thomas A. Courtade \citeyearpar{Courtade2017Concavity}}
\foundby{Suvrit Sra, by hand}{2022-05}
\formalizedby{Claude Fable 5}
\auditedby{Claude Fable 5}
\contributedby{Suvrit Sra}
\end{cxcredits}

\begin{cxcontext}
For two sets $K,L\subset\R^d$ the Minkowski sum is the set
$K+L :=\{x+y:x\in K,\ y\in L\}$. The statement below is the closing question of \citet{Courtade2017Concavity};
the ``Theorem 1'' refers to Courtade's similarly structured entropy-power inequality, of which he conjectured a volume analogue.  Dropping $V(K)^{1/d}V(L)^{1/d}$ weakens his question to the \emph{log-submodularity} of volume under Minkowski
summation, $V(B)V(K+L+B)\le V(K+B)V(L+B)$, a property with its own literature;
see the remarks after the proof. Following Courtade, we'll locally also use the notation $|K| \equiv V(K)$.
\end{cxcontext}

\begin{cxsource}{courtade-volume-conjecture}
Closing remarks (Section~IV) of \citet{Courtade2017Concavity}, posed as a
question for future work; no problem number is assigned there.
\end{cxsource}

\begin{cxstatement}{courtade-volume-conjecture}
In particular, for $K$, $L$, $B$ bounded convex sets in $\R^d$, does it hold
that:
\[
|K+L+B|^{1/d}|B|^{1/d}+|K|^{1/d}|L|^{1/d}\le|K+B|^{1/d}|L+B|^{1/d},
\]
where $|\cdot|$ denotes $d$-dimensional volume?
\end{cxstatement}

\begin{cxsummary}{courtade-volume-conjecture}
Courtade's geometric analogue of entropy-power concavity:
$|K{+}L{+}B|^{1/d}|B|^{1/d}+|K|^{1/d}|L|^{1/d}\le|K{+}B|^{1/d}|L{+}B|^{1/d}$
for bounded convex sets.
\end{cxsummary}

\begin{cxcertificate}{courtade-volume-conjecture}
A zonotope with seven integer generators and two segments; four integer
volumes with $68\cdot44=2992<3024=28\cdot108$.
\end{cxcertificate}

\begin{cxrefutation}
Problem~\ref{ctx:courtade-volume-conjecture} asks for the geometric analogue
of the entropy-power inequality that \citet{Courtade2017Concavity} proves for
Gaussian perturbations: volumes of Minkowski sums should play the role of
entropy powers of independent sums.  For $d=1$ the two sides agree
identically, and for $d=2$ the inequality is true
\citep{FradeliziMadimanMeyerZvavitch2024}.  It fails from dimension three
onwards.  We give a fully explicit witness in the zonotope calculus, where
every volume in sight is an integer.

\begin{theorem}[\statusfalse: Courtade's volume conjecture]
\label{thm:courtade-volume-conjecture}
For every $d\ge4$ there are bounded convex sets $K,L,B\subset\R^d$ with
\[
|K+L+B|^{1/d}|B|^{1/d}+|K|^{1/d}|L|^{1/d}>|K+B|^{1/d}|L+B|^{1/d}.
\]
One may take $B$ a zonotope and $K$, $L$ segments; in $\R^4$, $B$ has seven
integer generators and all three sets may moreover be taken full-dimensional.
\end{theorem}

\begin{proof}
A zonotope $Z(v_1,\dots,v_m)=\sum_{i\le m}[0,v_i]\subset\R^d$ has volume
$\sum_S|\det(v_S)|$, the sum running over all $d$-element subsets $S$ of the
generators \citep{Shephard1974}, and Minkowski summation simply concatenates
generator lists.  In $\R^4$ let $B$ be the zonotope with the seven columns of
\[
G=\begin{pmatrix}
1&0&0&0&0&1&0\\
0&1&0&0&1&0&0\\
0&0&1&0&-1&-1&1\\
0&0&0&1&1&1&1
\end{pmatrix}
\]
as generators, and let $K=[0,b]$ and $L=[0,c]$ be the segments with
$b=(1,-1,1,1)$ and $c=(0,0,-1,1)$.  Summing $|\det|$ over the
$\binom{7}{4}=35$ generator quadruples of $B$ (eleven are singular, twenty
have $|\det|=1$, four have $|\det|=2$) gives $|B|=28$, and likewise over the
$\binom{8}{4}=70$ quadruples of $B+K$ and of $B+L$ and the
$\binom{9}{4}=126$ of $B+K+L$:
\[
|B|=28,\qquad|K+B|=68,\qquad|L+B|=44,\qquad|K+L+B|=108.
\]
Since $|K|=|L|=0$, the conjectured inequality asserts
$(28\cdot108)^{1/4}\le(68\cdot44)^{1/4}$.  Raising to the fourth power, which
is monotone on nonnegative reals, this is $3024\le2992$ --- false.

For full-dimensional sets in $\R^4$, thicken each segment to a box: with
$\varepsilon=\tfrac1{100}$ replace $K$ by
$K_\varepsilon=[0,b]+[0,\varepsilon]^4$ and $L$ by
$L_\varepsilon=[0,c]+[0,\varepsilon]^4$, both zonotopes with five generators
and positive volume ($|K_\varepsilon|=\tfrac{401}{10^8}$,
$|L_\varepsilon|=\tfrac{201}{10^8}$).  The same generator sums give the exact
rationals
\[
|K_\varepsilon+B|=\tfrac{6926731601}{10^8},\qquad
|L_\varepsilon+B|=\tfrac{4484551401}{10^8},\qquad
|K_\varepsilon+L_\varepsilon+B|=\tfrac{696978401}{6250000},
\]
whence
$|B|\,|K_\varepsilon+L_\varepsilon+B|-|K_\varepsilon+B|\,|L_\varepsilon+B|
=\tfrac{161348459184476999}{10^{16}}>0$.  In fourth powers the conjectured
inequality fails a fortiori, its left side now carrying the additional
positive term $|K_\varepsilon|^{1/4}|L_\varepsilon|^{1/4}$.

For $d>4$ pad only the common summand: replace $B$ by
$B\times[0,1]^{d-4}$ and embed $K$, $L$ as the segments
$[0,(b,0)]$, $[0,(c,0)]$.  The three sums then acquire the same cube factor,
so all four volumes above are unchanged, and $|K|=|L|=0$ still; the same four
integers refute the inequality in $\R^d$.  Every quantity displayed in this
proof is an integer or an exact rational, recomputed by
\texttt{verify.py} in integer respectively rational arithmetic, and
cross-checked in SageMath, which rebuilds each Minkowski sum as a vertex
polytope over $\mathbb{Q}$ and computes its volume by triangulation instead.
\end{proof}

\noindent\textbf{The author's own (non AI) proof.}
Unlike most of this archive, this refutation was first derived by hand,
in May 2022 (independently of, and essentially contemporaneously with, the
arXiv posting of \citet{FradeliziMadimanZvavitch2024} in June 2022; no priority is claimed) by a three-step reduction that shows failure for every
$d\ge3$ without exhibiting a witness.  It is the shortest disproof of
Problem~\ref{ctx:courtade-volume-conjecture} we know of, so we record it in
full.  Throughout, $a,b,c,h$ are convex bodies in $\R^d$, $V(a)=|a|$, and
$V(a[d{-}k],h[k])$ is the mixed volume with $a$ repeated $d-k$ times.

\textcolor{darkblue}{\textbf{\emph{Step 1: drop a term.}}}  Applied to $(K,L,B)=(b,c,a)$ and stripped of the
nonnegative term $|b|^{1/d}|c|^{1/d}$, the conjecture asserts
$V(a)\,V(a+b+c)\le V(a+b)\,V(a+c)$, which rearranges to
\[
f_b(a+c)\ \ge\ f_b(a)
\qquad\text{for}\qquad
f_b:\ a\ \longmapsto\ \frac{V(a)}{V(a+b)}:
\]
for every fixed $b$, the functional $f_b$ must be non-decreasing along
Minkowski addition.

\textcolor{darkblue}{\textbf{\emph{Step 2: linearize.}}}  The expansion
$V(a+th)=\sum_{k\le d}\binom dk\,t^k\,V(a[d{-}k],h[k])$ is a polynomial in
$t\ge0$ with $\tfrac{d}{dt}V(a+th)\bigl|_{t=0}=d\,V(a[d{-}1],h)$, so
$t\mapsto\log f_b(a+th)$ is differentiable; by Step~1 (applied at each base
point $a+th$, with increment a multiple of $h$) it is non-decreasing, so its
derivative at $t=0$ is nonnegative:
\[
\frac{V(a[d{-}1],h)}{V(a)}\ \ge\ \frac{V((a{+}b)[d{-}1],h)}{V(a+b)}
\qquad\text{for all bodies }a,b,h.
\]

\textcolor{darkblue}{\textbf{\emph{Step 3: take $h$ a ball.}}}  Since $d\,V(a[d{-}1],B_2^d)=|\partial a|$ is
surface area, the display at $h=B_2^d$ is exactly the monotonicity
$I(a+b)\ge I(a)$ of the information functional
$I(\cdot)=|\cdot|/|\partial\,\cdot|$ of \citet{DemboCoverThomas1991}.
\citet[Section~4.1]{ArtsteinFlorentinOstrover2014} exhibit, for every $d\ge3$,
an interval $b$ and a truncated box $a$ with $I(a+b)<I(a)$ --- and prove that
monotonicity does hold for $d=2$, consistent with the planar case of the
conjecture being true.  Undoing Step~2, the strictly negative derivative gives
a $t>0$ with $V(a+b)\,V(a+tB_2^d)<V(a)\,V(a+b+tB_2^d)$, and the conjecture
fails at $(K,L,B)=(b,\,tB_2^d,\,a)$.

The explicit zonotope witness above was constructed for this archive in August
2026, so that the case is self-contained and exactly recomputable; it lands in
$d=4$ of necessity, since log-submodularity of volume holds for zonoids in
$d\le3$ \citep{FradeliziMadimanMeyerZvavitch2024}.

\vskip4pt
\noindent\textbf{What holds instead: hyperbolic polynomials}
Replace $\Vol^{1/d}$ on convex bodies by $p^{1/m}$ on the closed hyperbolicity cone
of a hyperbolic polynomial $p$ of degree $m$. Then, the four-point inequality becomes a theorem, since
Theorem~\ref{thm:lor-fourpoint} in Appendix~\ref{sec:lorentzianaudit} proves
\[
p^{1/m}(a+b)\,p^{1/m}(a+c)\ \ge\ p^{1/m}(a)\,p^{1/m}(a+b+c)+p^{1/m}(b)\,p^{1/m}(c),
\]
extra term included. The weaker inequality without the $p^{1/m}(b)\,p^{1/m}(c)$ was first observed by the author (S.~Sra) here~\citep{sra2017MO}. %

\vskip4pt
\noindent\textbf{Scope, and others' counterexamples.} The conjecture has by now been refuted several times over.  The sharp
Pl\"unnecke--Ruzsa constants of \citet{FradeliziMadimanZvavitch2024} already
imply that log-submodularity --- hence the stronger inequality here --- fails
for general convex bodies in every dimension $d\ge3$.  In August 2026,
\citet{Skorupinski2026} independently gave a zonotope counterexample to
log-submodularity in $\R^4$, with volume vector
$(|B|,|K{+}B|,|L{+}B|,|K{+}L{+}B|)=(14,30,26,56)$.  The witness above came
from an independent random search over integer generators, which turned up
dozens of violating configurations --- among them, rediscovered up to a
unimodular transformation, Skorupinski's own witness with its volume vector
$(14,30,26,56)$ --- so such violations are not isolated.  The configuration
kept here is a different one: its volume vector $(28,68,44,108)$ is not
proportional to his, so the two are not affinely equivalent, and its relative
margin is about twice as large.  What none of these witnesses
touch is the variant of the question in which the common summand $B$ is a
Euclidean ball --- the case closest to $W$ Gaussian in the entropy-power
inequality it mirrors, and one where the linearization above is provably
monotone \citep{DemboCoverThomas1991}.
\end{cxrefutation}

% ==== end counterexamples/courtade-volume-conjecture/case.tex ====
% ==== counterexamples/lorentzian-jensen/case.tex ====

\cxtitle{Log-volume midpoint gap and its Lorentzian generalization}

\begin{cxcredits}
\posedby{S.~Sra}
\foundby{GPT-5.6 Pro}{July 31, 2026}
\formalizedby{S.~Sra}
\auditedby{S.~Sra}
\contributedby{S.~Sra}
\end{cxcredits}

\begin{cxcontext}
\noindent\textbf{Note:} I could not locate the chat but I believe GPT-5.5 Pro already broke this earlier in 2026.
\vskip 8pt

For convex bodies, $K+L=\{k+\ell:k\in K,\ \ell\in L\}$ is the Minkowski sum and $\Vol$ is Lebesgue volume.  Recall that $d$ is a \emph{semimetric} when it is symmetric, nonnegative, vanishes only on equal arguments, and satisfies the triangle inequality; the triangle inequality is the property that fails below.  A homogeneous polynomial with nonnegative coefficients is \emph{Lorentzian} in the sense of Br\"and\'en and Huh \citeyearpar{BrandenHuh} if all its second-order directional derivatives, taken at positive arguments, are quadratic forms of Lorentzian signature $(+,-,\ldots,-)$; in two variables, a form $\sum_{k}a_kx^ky^{d-k}$ is strictly Lorentzian exactly when the normalized coefficients $a_k/\binom dk$ are strictly log-concave, which is the criterion used below.  Lorentzian polynomials are log-concave on the positive orthant, so $F=-\log G$ is convex there and every midpoint Jensen gap $\J_G$ is nonnegative.  On the cone of positive definite matrices the corresponding statement is a theorem: there $\J_{\det}$ is the squared S-divergence, whose square root is a metric \citep{SraSdiv}, and the same holds for a parameterized family of symmetric divergences on that cone \citep{SraJSMetrics2021}.
\end{cxcontext}

\begin{cxsource}{log-volume-distance}
S.~Sra, standalone problem posed in 2012; formulation as recorded in the audit of Appendix~\ref{sec:lorentzianaudit}
\end{cxsource}

\begin{cxstatement}{log-volume-distance}
The log-volume midpoint gap is the square of a distance on convex bodies: that is,
\[d_{\Vol}(K,L):=\left[\log\Vol\!\left(\frac{K+L}{2}\right)-\tfrac12\log\Vol(K)-\tfrac12\log\Vol(L)\right]^{1/2}\]
is a semimetric on convex bodies.
\end{cxstatement}

\begin{cxsummary}{log-volume-distance}
Log-volume midpoint gap as a squared distance on convex bodies \cite{Shephard}
\end{cxsummary}

\begin{cxcertificate}{log-volume-distance}
Strictly Lorentzian rational cubic realized as a Steiner polynomial in $\R^3$; exact rational triangle-violation bound $>7.15\times10^{-3}$.
\end{cxcertificate}

\begin{cxsource}{lorentzian-jensen}
S.~Sra, standalone problem posed in 2019 as a Lorentzian generalization of the log-volume statement; formulation as recorded in the audit of Appendix~\ref{sec:lorentzianaudit}
\end{cxsource}

\begin{cxstatement}{lorentzian-jensen}
Let $\mathcal C\subset\R^m$ be an open convex cone, let $G:\mathcal C\to(0,\infty)$, and put $F=-\log G$ and
\[\J_G(x,y):=\frac{F(x)+F(y)}2-F\!\left(\frac{x+y}{2}\right).\]
If $G$ is homogeneous and Lorentzian, then $d_G:=\sqrt{\J_G}$ is a semimetric.
\end{cxstatement}

\begin{cxsummary}{lorentzian-jensen}
Lorentzian Jensen--Bregman metric principle: $\sqrt{\J_G}$ a semimetric for every homogeneous Lorentzian $G$
\end{cxsummary}

\begin{cxcertificate}{lorentzian-jensen}
The same cubic read as a homogeneous Lorentzian polynomial; the violation transfers verbatim.
\end{cxcertificate}

\begin{cxrefutation}
Both statements fall to a single witness: a strictly Lorentzian bivariate cubic, carried to convex bodies by Shephard's realization theorem.  Throughout, $\J_G(x,y)=\log G(\frac{x+y}2)-\frac12\log G(x)-\frac12\log G(y)$ is the midpoint Jensen gap of $F=-\log G$, and $d_G=\sqrt{\J_G}$.  The witness is isolated first, so that each of the two statements below is a reading of the same three numbers.

\begin{lemma}[The witness]\label{lem:lorentzian-cubic}
Let
\begin{equation}\label{eq:lorentzian-cubic}
G(x,y)=y^3+\tfrac{11}5xy^2+\tfrac32x^2y+\tfrac1{10}x^3,\qquad(x,y)\in\R_{>0}^2,
\end{equation}
and take the rational points
\begin{equation}\label{eq:lorentzian-points}
p=(35,2),\qquad q=\left(\tfrac{11}2,\tfrac{23}2\right),\qquad r=\left(\tfrac1{500},15\right).
\end{equation}
Then $G$ is strictly Lorentzian, and
\begin{equation}\label{eq:lorentzian-gap}
d_G(p,r)-d_G(p,q)-d_G(q,r)>\frac{7156704176613847159228106087505875849130467}{10^{45}}>0.
\end{equation}
\end{lemma}
\begin{proof}
Writing $G=\sum_{k=0}^3a_kx^ky^{3-k}$ gives $(a_0,a_1,a_2,a_3)=(1,\frac{11}5,\frac32,\frac1{10})$.  By the bivariate criterion \cite[Example~2.3]{BrandenHuh}, strict Lorentzianity is strict log-concavity of $a_k/\binom3k$, and
\[
\left(1,\tfrac{11}{15},\tfrac12,\tfrac1{10}\right)
\text{ satisfies }
\left(\tfrac{11}{15}\right)^2=\tfrac{121}{225}>\tfrac12,\qquad
\left(\tfrac12\right)^2=\tfrac14>\tfrac{11}{150}.
\]
So $G$ is strictly Lorentzian; in particular $F=-\log G$ is convex, and every Jensen gap below is nonnegative.  With $R_{ab}:=G((a+b)/2)^2/(G(a)G(b))$ and $\J_G(a,b)=\frac12\log R_{ab}$, exact evaluation at \eqref{eq:lorentzian-points} gives
\[
R_{pq}=\frac{609584616081}{344696430080},\quad
R_{qr}=\frac{1254036986210090216149539001}{1235460453386242764000000000},\quad
R_{pr}=\frac{6141969899088096806341860001}{2794813396007162280000000000}.
\]
Outward-rounded rational enclosures for the three logarithms and their square roots---no floating point---then certify \eqref{eq:lorentzian-gap}.
\end{proof}

\begin{theorem}[\statusfalse: log-volume distance on convex bodies]\label{thm:volume-distance}
There are full-dimensional convex bodies $K,L\subset\R^3$ and Minkowski combinations $A,B,C$ of them with
\[
d_{\Vol}(A,C)>d_{\Vol}(A,B)+d_{\Vol}(B,C).
\]
Hence the proposed log-volume midpoint gap is not the square of a semimetric.
\end{theorem}
\begin{proof}
For convex bodies in $\R^3$ the relative Steiner polynomial is $\Vol(xK+yL)=\sum_{i=0}^3\binom3iW_i(K,L)x^{3-i}y^i$.  Put
\[
(W_0,W_1,W_2,W_3)=\left(\tfrac1{10},\tfrac12,\tfrac{11}{15},1\right),
\]
which satisfies the full set of Shephard inequalities in dimension three:
\[
W_1^2=\tfrac14>\tfrac{11}{150}=W_0W_2,\qquad
W_1W_2=\tfrac{11}{30}>\tfrac1{10}=W_0W_3,\qquad
W_2^2=\tfrac{121}{225}>\tfrac12=W_1W_3.
\]
By Shephard's realization theorem \citeyearpar[Theorem~4]{Shephard} these are the mixed volumes of some pair $K,L\subset\R^3$, whose Steiner polynomial is then exactly the cubic \eqref{eq:lorentzian-cubic}; the positive endpoint coefficients force $\Vol(K),\Vol(L)>0$, so both bodies are full-dimensional.  Setting $A=p_1K+p_2L$, $B=q_1K+q_2L$, $C=r_1K+r_2L$ for the points \eqref{eq:lorentzian-points}, Minkowski linearity gives $(A+B)/2=\frac{p_1+q_1}2K+\frac{p_2+q_2}2L$ and likewise for the other pairs, so
\[
d_{\Vol}(A,B)^2=\J_G(p,q),\qquad
d_{\Vol}(B,C)^2=\J_G(q,r),\qquad
d_{\Vol}(A,C)^2=\J_G(p,r),
\]
and the gap \eqref{eq:lorentzian-gap} of Lemma~\ref{lem:lorentzian-cubic} transfers verbatim.
\end{proof}

\begin{corollary}[\statusfalse: Lorentzian Jensen--Bregman metricity]\label{thm:lorentzian}
The cubic \eqref{eq:lorentzian-cubic} is a homogeneous strictly Lorentzian polynomial for which $d_G=\sqrt{\J_G}$ violates the triangle inequality on $\R_{>0}^2$.  Hence $\sqrt{\J_G}$ is not a semimetric for every homogeneous Lorentzian $G$.
\end{corollary}
\begin{proof}
Immediate from Lemma~\ref{lem:lorentzian-cubic}: the cubic \eqref{eq:lorentzian-cubic} is homogeneous of degree three and strictly Lorentzian, while \eqref{eq:lorentzian-gap} is exactly the failure of the triangle inequality for $d_G$ at the three points \eqref{eq:lorentzian-points}.
\end{proof}

\begin{remark}[What the audit leaves standing]
The midpoint integral identity and the homogeneous primitive decomposition of the earlier positive argument are correct.  What fails is a claimed Gram-determinant identity for the derivative of a Hessian Rayleigh quotient, whose two sides have incompatible homogeneity in the line direction; the omitted calculation therefore conceals a false identity rather than a technical gap.  Appendix~\ref{sec:lorentzianaudit} carries the full audit, including the exact enclosure procedure behind \eqref{eq:lorentzian-gap}.
\end{remark}
\end{cxrefutation}

% ==== end counterexamples/lorentzian-jensen/case.tex ====
% ==== counterexamples/dpp-feasible-step/case.tex ====

\cxtitle{Feasible Picard steps for DPP likelihood}

\begin{cxcredits}
\posedby{Mariet and Sra \citeyearpar{MarietSra2015}}
\foundby{GPT-5.6}{Aug 2026}
\formalizedby{S.~Sra}
\auditedby{S.~Sra}
\contributedby{S.~Sra}
\end{cxcredits}

\begin{cxcontext}
A determinantal point process with positive definite kernel $L$ on a ground set $\mathcal Y$ assigns a subset $Y\subseteq\mathcal Y$ probability proportional to $\det(L_Y)$, the principal minor on $Y$.  Given observed subsets $Y_1,\ldots,Y_n$, learning the kernel $L$ means maximizing the log-likelihood
\[
\phi(L)=\sum_{i=1}^n\log\det(L_{Y_i})-n\log\det(I+L),
\]
whose stationarity $\nabla\phi(L)=0$ can be written as $\Delta+L^{-1}=L^{-1}$ for the residual
\[
\Delta=\frac1n\sum_{i=1}^nU_i(U_i^*LU_i)^{-1}U_i^*-(I+L)^{-1},
\]
where $U_i$ is the indicator matrix selecting the coordinates in $Y_i$.  The stationarity suggests the Picard iteration $L_{k+1}=L_k+aL_k\Delta_kL_k$ of the quoted passage; a step size $a$ is \emph{feasible} when it keeps the iterate positive definite, which the source guarantees for $a\le1/(1-\gamma)$ with $\gamma=\max\{\lambda_{\min}(LZ),1/\lambda_{\max}(I+L)\}\in[0,1]$.  Ascent is proved for $a=1$; the conjecture asks for every feasible $a\ge1$.
\end{cxcontext}

\begin{cxsource}{dpp-feasible-step}
Z.~Mariet and S.~Sra \citeyearpar{MarietSra2015}, \S2 (unnumbered)
\end{cxsource}

\begin{cxstatement}{dpp-feasible-step}
Above we showed that for $a=1$ ascent is guaranteed.  Empirically, $a>1$ often works well; Prop.~A.1 presents an easily computable upper bound on feasible $a$.  We conjecture that for all feasible values $a\ge1$, iteration (2.5) is guaranteed to increase the log-likelihood.

The iteration carrying the step size is $L_{k+1}=L_k+aL_k\Delta_kL_k$, displayed as (2.7) in the source; feasibility is the bound $a\le1/(1-\gamma)$ of Prop.~A.1.
\end{cxstatement}

\begin{cxsummary}{dpp-feasible-step}
Mariet--Sra feasible-step Picard ascent conjecture for DPP likelihood \citeyearpar{MarietSra2015}
\end{cxsummary}

\begin{cxcertificate}{dpp-feasible-step}
Exact rational $2\times2$ feasible update with lower likelihood.
\end{cxcertificate}

\begin{cxrefutation}
\begin{theorem}[\statusfalse: feasible-step ascent conjecture]\label{thm:dpp}
A feasible step with $a=5$ can strictly decrease the DPP log-likelihood.
\end{theorem}
\begin{proof}
Use ground set $\{1,2\}$ and observations $\mathcal T=\bigl\{\{1\},\{2\},\{1,2\}\bigr\}$.
Let
\[
L_0=\begin{pmatrix}3.37200647&2.62325460\\2.62325460&6.07953037\end{pmatrix},
\]
with the displayed decimals interpreted as exact rationals of denominator $10^8$.  For this data,
\[
\Delta(L)=\frac13\left[
\begin{pmatrix}L_{11}^{-1}&0\\0&0\end{pmatrix}
+\begin{pmatrix}0&0\\0&L_{22}^{-1}\end{pmatrix}+L^{-1}\right]-(I+L)^{-1}.
\]
At $a=5$, exact rational arithmetic gives
\[
L_1=L_0+5L_0\Delta(L_0)L_0
\approx\begin{pmatrix}
3.1679125391632645&3.1681701017272776\\
3.1681701017272776&3.4400758121589554
\end{pmatrix}.
\]
Both matrices are positive definite; numerically, $\det L_0\approx13.6187510458$ and $\det L_1\approx0.8605575075$.
The normalized objective is $\phi(L)=\frac13\bigl(\log L_{11}+\log L_{22}+\log\det L\bigr)-\log\det(I+L).$
The inequality $\phi(L_1)<\phi(L_0)$ is equivalent, after exponentiating and cubing, to the exact rational comparison
\[
\frac{(L_1)_{11}(L_1)_{22}\det L_1}{(L_0)_{11}(L_0)_{22}\det L_0}
<\left(\frac{\det(I+L_1)}{\det(I+L_0)}\right)^3.
\]
Numerically, the left and right sides are respectively
\[
0.03359118941299074
\quad\text{and}\quad
0.04354945027745948.
\]
Their exact rational difference is positive.  Hence the step is feasible but descending.
\end{proof}
\end{cxrefutation}

% ==== end counterexamples/dpp-feasible-step/case.tex ====
% ==== counterexamples/sdd-nystrom-diminishing-returns/case.tex ====
\cxtitle{Strict diagonal dominance does not ensure diminishing Nystr\"om error reductions}

\begin{cxcredits}
\posedby{Mark Fornace \citeyearpar{AmselEtAl2026}}
\foundby{OpenAI Codex}{2026-08}
\formalizedby{OpenAI Codex}
\auditedby{OpenAI Codex}
\contributedby{Suvrit Sra}
\end{cxcredits}

\begin{cxcontext}
For a symmetric positive-definite matrix $L$ and $\gamma>0$, put
$K=(L+\gamma I)^{-1}$.  For a nonempty selected index set $\mathcal I$, the
nuclear Nystr\"om error is
\[
 F(\mathcal I)=\left\|K-K_{:,\mathcal I}
 K_{\mathcal I,\mathcal I}^{-1}K_{\mathcal I,:}\right\|_*.
\]
The residual is positive semidefinite, so its nuclear norm equals its trace.
The intended diminishing-error-reduction property is that, for
$S\subseteq T$ and $i\notin T$,
\[
 F(S)-F(S\cup\{i\})\geq F(T)-F(T\cup\{i\}).
\]
This formulation avoids the convention-dependent choice between the words
``submodular'' and ``supermodular'' for a decreasing error function.
The property is what would license selecting $\mathcal I$ greedily; without it,
landmark sets are chosen by other means, for instance by determinantal sampling
\citep{LiJegelkaSra2016} or through elementary-symmetric-polynomial design
objectives \citep{MarietSra2017}.
Colbrook had already resolved Problem~\ref{ctx:sdd-nystrom-diminishing-returns}
in July 2026, proving the SDDM case
and giving exact SDD obstructions \cite{Colbrook2026}.  The matrix below was
found independently in August 2026; no priority over Colbrook is claimed.
\end{cxcontext}

\begin{cxsource}{sdd-nystrom-diminishing-returns}
Problem~4.6 of Amsel et al. \cite{AmselEtAl2026}, posed there by Mark Fornace.
Colbrook \cite{Colbrook2026} subsequently completed the answer before this
independent witness was found.
\end{cxsource}

\begin{cxstatement}{sdd-nystrom-diminishing-returns}
Prove or disprove the submodularity of the nuclear Nystr\"om error~(14) when
$L$ is assumed, in contrast to the above, to be (a) SDDM and positive-definite
or (b) SDD and positive-definite.
\end{cxstatement}

\begin{cxsummary}{sdd-nystrom-diminishing-returns}
A strictly diagonally dominant positive-definite $4\times4$ signed matrix has
a nonempty-base marginal Nystr\"om error reduction that increases after an
additional index is selected.
\end{cxsummary}

\begin{cxcertificate}{sdd-nystrom-diminishing-returns}
An exact integer $4\times4$ matrix, rational shift, four rational principal
inverse traces, and the positive violation gap $2165/14833896$.
\end{cxcertificate}

\begin{cxrefutation}
Problem~\ref{ctx:sdd-nystrom-diminishing-returns} asks whether the nuclear
Nystr\"om error retains diminishing reductions for positive-definite SDDM or
SDD matrices \cite{AmselEtAl2026}.
Part~(a) is true, while part~(b) is false \cite{Colbrook2026}.  Here is an
independently found exact witness for part~(b) with a nonempty base set.

\begin{theorem}[\statusfalse: SDD Nystr\"om diminishing returns]
\label{thm:sdd-nystrom-diminishing-returns}
The nuclear Nystr\"om error need not have diminishing error reductions when
$L$ is symmetric diagonally dominant and positive definite, even when $L$ is
strictly diagonally dominant and the comparison uses a nonempty base set.
\end{theorem}
\begin{proof}
Let
\[
 M=\begin{pmatrix}
 11&3&-4&3\\
 3&12&4&-1\\
 -4&4&14&-2\\
 3&-1&-2&11
 \end{pmatrix},\qquad
 \gamma=\frac12,\qquad L=M-\frac12I,
\]
and set $K=(L+\gamma I)^{-1}=M^{-1}$.  The diagonal entries of $L$ minus the
sums of the absolute values of the corresponding off-diagonal entries are
\[
 \frac12,\quad\frac72,\quad\frac72,\quad\frac92.
\]
Thus $L$ is symmetric strictly diagonally dominant with positive diagonal;
Gershgorin's theorem gives $L\succ0$.

For a selected set $\mathcal I$ with complement $J$, the block inverse
identity gives
\[
 K_{J,J}-K_{J,\mathcal I}K_{\mathcal I,\mathcal I}^{-1}
 K_{\mathcal I,J}=(M_{J,J})^{-1}.
\]
All other blocks of the Nystr\"om residual vanish.  Consequently
\[
 F(\mathcal I)=\operatorname{Tr}\bigl((M_{J,J})^{-1}\bigr).
\]

Use one-based indices, take the nonempty base $S=\{2\}$, and compare adding
$i=3$ before and after $j=4$.  Direct exact inversion of the four relevant
principal submatrices gives
\[
 F(\{2\})=\frac{100}{349},\qquad
 F(\{2,3\})=\frac{11}{56},\qquad
 F(\{2,4\})=\frac{25}{138},\qquad
 F(\{2,3,4\})=\frac1{11}.
\]
For example, the complement of $\{2\}$ is indexed by $\{1,3,4\}$; its
$3\times3$ principal submatrix has determinant $1396$ and the sum of its
three diagonal cofactors is $400$, yielding $400/1396=100/349$.

The two marginal error reductions are therefore
\[
 F(\{2\})-F(\{2,3\})=\frac{1761}{19544},\qquad
 F(\{2,4\})-F(\{2,3,4\})=\frac{137}{1518}.
\]
Their difference in the forbidden direction is
\[
 \frac{137}{1518}-\frac{1761}{19544}
 =\frac{2165}{14833896}>0.
\]
Hence selecting index $3$ reduces the error more, not less, after index $4$
has already been selected.  This exactly violates diminishing error
reductions.  Every displayed quantity is an integer or an exact rational.
\end{proof}
\end{cxrefutation}
% ==== end counterexamples/sdd-nystrom-diminishing-returns/case.tex ====

\subsection{Statements recorded without a dossier}

These fell to the same standard of evidence as every case above --- each carries
a certificate in the repository that recomputes in exact arithmetic --- but the
statements are slight enough that the archive records them here rather than
spending a section and a ledger row on each.

\medskip
\begin{itemize}[itemsep=0.4em]
% ==== tex/generated/appendix-brief.tex ====
\cxbriefitem{Problem 5.1}{A one-dimensional oblivious subspace injection with injectivity $1$ can make sketch-and-solve incur the fixed factor $\sqrt2$ with probability $1/50$.}{Three exact rationally weighted $2\times2$ sketch atoms, with exact isotropy, an exact $99/100$ injection guarantee, and exact squared residual ratios.}{counterexamples/osi-sketch-and-solve/}
\cxbriefitem{the identity after eq. (20) in §6.2}{The proposed NEPv matrix does not reproduce the product-state Rayleigh quotient when the Hamiltonian sum includes a same-site term with $i=j$.}{The certificate consists of exact $2\times2$ integer matrices and the integer vector $(1,2)^T$, for which the two assertedly equal values are $0$ and $-4/25$.}{counterexamples/hamiltonian-nepv-identity/}
\cxbriefitem{Problem 4.3}{For an exact rational rank-three projector on eight coordinates, QRCP has strict pivots $I=(1,2,3)$ but the resulting inverse norm exceeds $\sqrt{18}$.}{An exact rational $8\times3$ Grassmann-chart matrix, three positive rational pivot gaps, and the Rayleigh gap $192051/250000$.}{counterexamples/qrcp-orthonormal-greedy/}
% ==== end tex/generated/appendix-brief.tex ====
\end{itemize}

% ==== end tex/09-additional.tex ====

% ==== tex/03-proof-details.tex ====
\section{Positive results for Sah et al's mixed-norm CS inequality}
\label{sec:proofdetails}

This appendix collects the positive results that bound where the mixed-norm
counterexamples of Theorems~\ref{thm:mixed-general} and~\ref{thm:mixed} can live.
Throughout we use the convention of \citet{SahSawhneyStonerZhao2020},
\[
\|X\|_{{p,q}}:=\Bigl(\nlsum_i\Bigl(\nlsum_j|x_{ij}|^p\Bigr)^{q/p}\Bigr)^{1/q},
\]
which is the transpose of the convention used in the reported counterexamples; on the
symmetric matrices $A^{\mathsf T}A$, $B^{\mathsf T}B$ and at $p=q$ the two agree,
which covers every statement below.  Write $p'$ for the conjugate exponent,
$1/p+1/p'=1$.  The question is whether
\begin{equation}\label{eq:ssszq}
\|A^{\mathsf T}B\|_{{p,q}}^2\ \le\ \|A^{\mathsf T}A\|_{{q,q}}\,\|B^{\mathsf T}B\|_{{p,p}}
\end{equation}
holds for all real $1\le p\le q$ and all entrywise nonnegative $A,B$.

\begin{lemma}\label{lem:dualitybound}
Let $M$ be Hermitian and $q\ge1$.  Then, $\sup_{\|w\|_{q'}\le1}w^*Mw\le\|M\|_{{q,q}}$.
\end{lemma}
\begin{proof}
Since $\|w\|_{q'}\le1$ iff $\|w\|_{q'}^2\le1$, and $\|w\|_{q'}^2=\|ww^*\|_{{q',q'}}$,
\[
\sup_{\|w\|_{q'}\le1}w^*Mw=\sup_{\|ww^*\|_{{q',q'}}\le1}\tr(Mww^*)
\le\sup_{\|W\|_{{q',q'}}\le1}\tr(MW)=\|M\|_{{q,q}},
\]
the last step being duality of $\|\cdot\|_{{q,q}}$ and $\|\cdot\|_{{q',q'}}$.
\end{proof}

\begin{proposition}\label{prop:pone}
Let $A,B$ be entrywise nonnegative and $q\ge1$.  Then \eqref{eq:ssszq} holds at $p=1$.
\end{proposition}
\begin{proof}
As $A^{\mathsf T}B\ge0$ entrywise, $\|A^{\mathsf T}Be_i\|_1=\mathbf 1^{\mathsf T}A^{\mathsf T}Be_i$, so by duality
\[
\|A^{\mathsf T}B\|_{{1,q}}=\sup_{\|u\|_{q'}\le1}\sum_i u_i\mathbf 1^{\mathsf T}A^{\mathsf T}Be_i
=\sup_{\|u\|_{q'}\le1}\langle A\mathbf 1,\ Bu\rangle
\le\sup_{\|u\|_{q'}\le1}\|A\mathbf 1\|_2\|Bu\|_2 ,
\]
by Cauchy--Schwarz.  Now $\|A\mathbf 1\|_2=(\mathbf 1^{\mathsf T}A^{\mathsf T}A\mathbf 1)^{1/2}=\|A^{\mathsf T}A\|_{L_{1,1}}^{1/2}$, while
$$\sup_{\|u\|_{q'}\le1}\|Bu\|_2=\sup_{\|u\|_{q'}\le1}(u^{\mathsf T}B^{\mathsf T}Bu)^{1/2}\le\|B^{\mathsf T}B\|_{{q,q}}^{1/2}$$
by Lemma~\ref{lem:dualitybound}.
\end{proof}

This result appears as Lemma~3.1 in~\citet{SahSawhneyStonerZhao2020}; the duality proof recorded
there was supplied by S.~Sra, and the paper's acknowledgment records as much.

\begin{proposition}[every positive integer $p$]\label{prop:intp}
If $p\in\mathbb N$ and $q\ge p$, then \eqref{eq:ssszq} holds for all entrywise nonnegative $A,B$.
\end{proposition}
\begin{proof}
The block matrix $\begin{bsmallmatrix}A^{\mathsf T}A&A^{\mathsf T}B\\B^{\mathsf T}A&B^{\mathsf T}B\end{bsmallmatrix}=[A\ B]^{\mathsf T}[A\ B]$ is positive semidefinite, so by the Schur product theorem its $p$-th Hadamard power is too, and is therefore $\begin{bsmallmatrix}C^{\mathsf T}C&C^{\mathsf T}D\\D^{\mathsf T}C&D^{\mathsf T}D\end{bsmallmatrix}$ for some $C,D$ with $C^{\mathsf T}D$, $C^{\mathsf T}C$, $D^{\mathsf T}D$ entrywise nonnegative.  Writing $r=q/p\ge1$,
\[
\|A^{\mathsf T}B\|_{{p,q}}=\Bigl(\sum_i\Bigl(\sum_j (A^{\mathsf T}B)_{ij}^{p}\Bigr)^{q/p}\Bigr)^{1/q}
=\|C^{\mathsf T}D\|_{{1,r}}^{1/p}
\le\bigl(\|C^{\mathsf T}C\|_{{1,1}}^{1/2}\|D^{\mathsf T}D\|_{{r,r}}^{1/2}\bigr)^{1/p},
\]
by Proposition~\ref{prop:pone}, and the right-hand side is $\|A^{\mathsf T}A\|_{{p,p}}^{1/2}\|B^{\mathsf T}B\|_{{q,q}}^{1/2}$.
\end{proof}

\begin{proposition}[$q=\infty$]\label{prop:qinf}
For $1\le p$ and any complex $A,B$ of compatible size,
$\|A^*B\|_{{p,\infty}}^2\le\|A^*A\|_{{p,p}}\|B^*B\|_{{\infty,\infty}}$.
Here entrywise nonnegativity is not needed.
\end{proposition}
\begin{proof}
By duality and Cauchy--Schwarz,
\[
\|A^*B\|_{{p,\infty}}^2=\max_i\max_{\|z_i\|_{p'}\le1}|z_i^*A^*Be_i|^2
\le\max_i\max_{\|z_i\|_{p'}\le1}(z_i^*A^*Az_i)\,(B^*B)_{ii},
\]
and Lemma~\ref{lem:dualitybound} bounds the first factor by $\|A^*A\|_{{p,p}}$ and the second by $\|B^*B\|_{{\infty,\infty}}$.
\end{proof}

The next two statements concern the specialization $B=A$, i.e.\ the inequality
$\|Z\|_{{p,q}}^2\le\|Z\|_{{p,p}}\|Z\|_{{q,q}}$ for completely positive $Z$.

\begin{proposition}[infinite divisibility]\label{prop:infdiv}
Let $Z\in\R_{\ge0}^{n\times n}$ be positive semidefinite and suppose the Hadamard power $Z^{\odot p}=[z_{ij}^p]$ is positive semidefinite for some $p\ge1$.  Then $\|Z\|_{{p,q}}^2\le\|Z\|_{{p,p}}\|Z\|_{{q,q}}$ for every $q\ge p$.
\end{proposition}
\begin{proof}
Write $Z^{\odot p}=X^{\mathsf T}X$.  Then $\|Z\|_{{p,q}}=\|X^{\mathsf T}X\|_{{1,q/p}}^{1/p}$, and Proposition~\ref{prop:pone} bounds this by $\bigl(\|X^{\mathsf T}X\|_{{1,1}}\|X^{\mathsf T}X\|_{{q/p,q/p}}\bigr)^{1/2p}=\|Z\|_{{p,p}}^{1/2}\|Z\|_{{q,q}}^{1/2}$.
\end{proof}

\begin{corollary}[order at most three]\label{cor:smalln}
Let $Z\in\R^{n\times n}$ be positive semidefinite with $n\le3$, with no assumption that its entries are nonnegative.  Then $\|Z\|_{{p,q}}^2\le\|Z\|_{{p,p}}\|Z\|_{{q,q}}$ for all $1\le p\le q$.
\end{corollary}
\begin{proof}
For $n\le3$ a PSD $Z$ has $|Z|:=[|z_{ij}|]$ PSD as well.\footnote{This is a classical fact about small Hermitian matrices; the source manuscript for this case attributes it to Marcus and Watkins (1969), a reference the present text has not independently checked.}  Since $p\ge1\ge n-2$, the FitzGerald--Horn theorem \citep{FitzGeraldHorn1977} makes $|Z|^{\odot p}$ positive semidefinite, so Proposition~\ref{prop:infdiv} applies to $|Z|$.  The mixed norms depend only on $|z_{ij}|$, so the bound transfers to $Z$.
\end{proof}

Together, Propositions~\ref{prop:pone}, \ref{prop:intp}, \ref{prop:qinf} and Corollary~\ref{cor:smalln} confine any counterexample
to non-integer $p$, finite $q$, and order at least four.  Theorem~\ref{thm:mixed-general}
lives at $p=q=3/2$ and order four (as a Gram matrix of two $3\times2$ blocks), and
Theorem~\ref{thm:mixed} at $(6/5,6)$ and order $21$ --- in each case the smallest region the
positive results leave open.

% ==== end tex/03-proof-details.tex ====

% ==== tex/04-aim36-pencil.tex ====
\section{Specht modules and Borcea--Br\"and\'en's Problems 36, 37}
\label{sec:aim36pencil}

The integer matrices of Theorems~\ref{thm:aim36} and~\ref{thm:aim37} are not a numerical accident,
and this appendix says where they come from.  They are the $t=5$ and $t=4$
members of one family of pencils built equivariantly on the Specht module
$S^{(3,2)}$, whose $s_{(1^5)}$ coefficient is a fixed quintic in $t$ that turns
negative once $t$ is large enough.  Both problems are refuted by that single
coefficient, but at different values of $t$, and the appendix ends by locating
the two thresholds.

\subsection*{The module and its projections}

Let $M$ be the real permutation module with orthonormal basis
$\{e_{ab}:1\le a<b\le5\}$ indexed by the $2$-subsets of $[5]$, and let
\[
B:M\longrightarrow\R^5,\qquad B(e_{ab})=e_a+e_b
\]
be the $\mathfrak S_5$-equivariant map that forgets which pair produced which
points.  Row $i$ of $B$ has a $1$ in the four columns $ab$ with $i\in\{a,b\}$,
and two rows share exactly one such column, so
\[
BB^{\mathsf T}=3I_5+\mathbf 1\mathbf 1^{\mathsf T},
\]
which is nonsingular.  Hence $B$ is surjective and $V:=\ker B$ has dimension
$10-5=5$.  Young's rule for the permutation module on $2$-subsets gives
$M\cong S^{(5)}\oplus S^{(4,1)}\oplus S^{(3,2)}$, while the point-permutation
module $\R^5$ is $S^{(5)}\oplus S^{(4,1)}$, so
\[
V\cong S^{(3,2)}.
\]
Write $\rho$ for the representation of $\mathfrak S_5$ on $V$, orthogonal for
the inner product $V$ inherits from $M$.  For $i\in[5]$ set
\begin{equation}\label{eq:aim36-Pi}
P_i:=\frac12\sum_{\substack{1\le a<b\le5\\a,b\ne i}}\rho((ab)),
\end{equation}
half the sum of the transpositions fixing $i$.

\begin{lemma}\label{lem:aim36-proj}
Each $P_i$ is an orthogonal projection of rank three,
$\rho(\sigma)P_i\rho(\sigma)^{-1}=P_{\sigma(i)}$ for every
$\sigma\in\mathfrak S_5$, and $\sum_{i=1}^5P_i=3I_V$.
\end{lemma}
\begin{proof}
The transpositions in \eqref{eq:aim36-Pi} are exactly those of the point
stabilizer $H_i\cong\mathfrak S_4$, and their sum is the class sum of
transpositions, which is central in $\R[\mathfrak S_4]$ and therefore acts on an
irreducible Specht module $S^\mu$ by the sum of the contents of $\mu$.  By
branching,
\[
S^{(3,2)}\!\downarrow_{\mathfrak S_4}\cong S^{(3,1)}\oplus S^{(2,2)},
\]
and the content sums are $0+1+2-1=2$ for $(3,1)$ and $0+1-1+0=0$ for $(2,2)$.
So the operator in \eqref{eq:aim36-Pi} is the identity on the three-dimensional
$S^{(3,1)}$ summand and zero on the two-dimensional $S^{(2,2)}$ summand: an
orthogonal projection of rank three.  Covariance follows by conjugating the
class sum, since conjugation by $\rho(\sigma)$ carries the transpositions
fixing $i$ to those fixing $\sigma(i)$.  Finally $\sum_iP_i$ commutes with
$\rho(\mathfrak S_5)$, so Schur's lemma makes it a scalar on the irreducible
$V$; its trace is $5\cdot3=15$, hence it is $3I_V$.
\end{proof}

\subsection*{The one-parameter family}

For $t\ge0$ put
\[
A_i(t):=I_V+tP_i,\qquad
p_t(x):=\det\Bigl(\nlsum_{i=1}^5x_iA_i(t)\Bigr),
\]
the determinant taken on $V$.  By Lemma~\ref{lem:aim36-proj} each $A_i(t)$ is
symmetric positive definite with eigenvalues $1+t$ of multiplicity three and
$1$ of multiplicity two.

\begin{proposition}\label{prop:aim36-family}
For every $t\ge0$ the polynomial $p_t$ is symmetric, homogeneous of degree five,
real stable, and has nonnegative coefficients in the monomial basis.  Its Schur
coefficients are
\begin{align*}
[s_{(5)}]p_t&=(t+1)^3,\\
[s_{(4,1)}]p_t&=\tfrac12(t+1)^2(3t^2+8t+8),\\
[s_{(3,2)}]p_t&=\tfrac1{16}(t+1)(9t^4+48t^3+128t^2+160t+80),\\
[s_{(3,1,1)}]p_t&=\tfrac1{16}(t+1)(9t^4+64t^3+168t^2+192t+96),\\
[s_{(2,2,1)}]p_t&=\tfrac18(6t^5+41t^4+108t^3+156t^2+120t+40),\\
[s_{(2,1,1,1)}]p_t&=\tfrac18(6t^5+35t^4+96t^3+132t^2+96t+32),\\
[s_{(1^5)}]p_t&=-\tfrac1{16}(9t^5-21t^4-56t^3-72t^2-48t-16).
\end{align*}
In particular $[s_{(1^5)}]p_5=-743/2<0$.
\end{proposition}
\begin{proof}
Symmetry is the covariance of Lemma~\ref{lem:aim36-proj}: relabeling the variables
by $\sigma$ conjugates the pencil by $\rho(\sigma)$, which leaves its
determinant alone.  For real stability, suppose $z_1,\ldots,z_5$ lie in the open
upper half-plane and $\bigl(\sum_iz_iA_i(t)\bigr)v=0$ for some $v\ne0$.  Then
\[
0=\operatorname{Im}\Bigl\langle v,\Bigl(\nlsum_iz_iA_i(t)\Bigr)v\Bigr\rangle
 =\nlsum_i\operatorname{Im}(z_i)\,\langle v,A_i(t)v\rangle>0,
\]
because each $A_i(t)$ is positive definite --- a contradiction.  Homogeneity of
degree five is immediate.  The coefficient of $x_1^{\alpha_1}\cdots
x_5^{\alpha_5}$ is a positive multiple of the mixed discriminant of the multiset
holding $\alpha_i$ copies of $A_i(t)$, and mixed discriminants of positive
semidefinite matrices are nonnegative, so the monomial coefficients are
nonnegative --- here in fact strictly positive.

For the coefficients themselves, expanding the determinant exactly gives
$p_t=\sum_{\mu\vdash5}c_\mu(t)m_\mu$ with
\begin{align*}
c_{(5)}(t)&=(t+1)^3,\\
c_{(4,1)}(t)&=\tfrac12(t+1)^2(3t^2+10t+10),\\
c_{(3,2)}(t)&=\tfrac1{16}(t+1)(9t^4+72t^3+232t^2+320t+160),\\
c_{(3,1,1)}(t)&=\tfrac18(t+1)(t+2)(9t^3+62t^2+120t+80),\\
c_{(2,2,1)}(t)&=\tfrac1{16}(39t^5+317t^4+1040t^3+1728t^2+1440t+480),\\
c_{(2,1,1,1)}(t)&=\tfrac18(45t^5+348t^4+1100t^3+1764t^2+1440t+480),\\
c_{(1^5)}(t)&=\tfrac18(99t^5+765t^4+2320t^3+3600t^2+2880t+960),
\end{align*}
every numerator having positive coefficients, which is the promised positivity
again.  In the partition order
$(5),(4,1),(3,2),(3,1,1),(2,2,1),(2,1,1,1),(1^5)$ the Kostka matrix of
$s_\lambda=\sum_\mu K_{\lambda\mu}m_\mu$ is
\[
K=\begin{pmatrix}
1&1&1&1&1&1&1\\
0&1&1&2&2&3&4\\
0&0&1&1&2&3&5\\
0&0&0&1&1&3&6\\
0&0&0&0&1&2&5\\
0&0&0&0&0&1&4\\
0&0&0&0&0&0&1
\end{pmatrix},
\]
and if $c$ and $a$ are the row vectors of monomial-symmetric and Schur
coefficients then $c=aK$.  Applying $K^{-1}$ to the $c_\mu(t)$ gives the table
above.  The last line needs only the final column of $K^{-1}$,
$(1,-2,-2,3,3,-4,1)^{\mathsf T}$:
\begin{align*}
[s_{(1^5)}]p_t&=c_{(5)}-2c_{(4,1)}-2c_{(3,2)}+3c_{(3,1,1)}+3c_{(2,2,1)}
  -4c_{(2,1,1,1)}+c_{(1^5)}\\
&=-\tfrac1{16}(9t^5-21t^4-56t^3-72t^2-48t-16).
\end{align*}
At $t=5$ this is $-(28125-13125-7000-1800-240-16)/16=-5944/16=-743/2$.
\end{proof}

\begin{remark}[Two thresholds]
The family crosses the two problems at different places, and the gap between
them is where Theorem~\ref{thm:aim37} lives.  Schur positivity survives exactly as long
as $[s_{(1^5)}]p_t\ge0$, that is until the single positive real zero
$t_2\approx4.2907656153$ of $9t^5-21t^4-56t^3-72t^2-48t-16$; past $t_2$ the
family refutes Problem~\ref{ctx:aim-problem-36}.  The lower endpoint of
Problem~\ref{ctx:aim-problem-37} goes earlier.  Since $f^{(1^5)}=1$ and
$[s_{(5)}]p_t=(t+1)^3$, its margin at $\lambda=(1^5)$ is
\[
(t+1)^3-[s_{(1^5)}]p_t=\frac{t^2\bigl(9t^3-21t^2-40t-24\bigr)}{16},
\]
so the endpoint fails exactly past the cubic's single positive zero
$t_1\approx3.7205453558$.  On $(t_1,t_2)$ the family is therefore still Schur
positive and already short of the endpoint --- the substantive refutation of
Problem~\ref{ctx:aim-problem-37} --- and $t=4$ is the only integer there, which
is why Theorem~\ref{thm:aim37} is stated at it.  Beyond $t_2$ the two failures merge and
stop being distinguishable, which is the reason $q$ at $t=5$ cannot serve as a
witness for the endpoint.  Both decimals are stated for orientation only: no step
of the proof, and no assertion in the certificate, uses a floating-point value.
\end{remark}

\subsection*{Back to the integer pencil}

What remains is bookkeeping.  Choose a rational basis of $V=\ker B$, assemble it
as the columns of a matrix $U$, and let $G=U^{\mathsf T}U$ be its Gram matrix.
In that basis $P_i$ is no longer symmetric but $G$-self-adjoint,
$P_i^{\mathsf T}G=GP_i$, so with
\[
J_i(t):=2G\bigl(I+tP_i\bigr)
\]
we get $J_i(t)^{\mathsf T}=2(I+tP_i)^{\mathsf T}G=2G(I+tP_i)=J_i(t)$: an ordinary
symmetric matrix, positive definite because $A_i(t)$ is, and at $t=4$ and $t=5$
an integer one in the basis the certificate picks.  Since $\det(G)=162$ there,
\[
\det\Bigl(\nlsum_ix_iJ_i(t)\Bigr)=2^5\det(G)\,p_t(x)=5184\,p_t(x).
\]
At $t=5$, writing $J_i:=J_i(5)$ and dividing by $648$ leaves
\[
q(x)=\frac1{648}\det\Bigl(\nlsum_ix_iJ_i\Bigr)=8\,p_5(x),
\]
which is the normalization in \eqref{eq:aim36-q}, and
$[s_{(1^5)}]q=8\cdot(-743/2)=-2972$, the coefficient of Theorem~\ref{thm:aim36}.  The
factor $8$ is what clears the denominators of
Proposition~\ref{prop:aim36-family} and leaves that whole Schur expansion
integral.  At $t=4$ no such factor is needed --- $p_4$ is already integral --- so
$\widetilde J_i:=J_i(4)$ and
\[
\widetilde q(x)=\frac1{5184}\det\Bigl(\nlsum_ix_i\widetilde J_i\Bigr)=p_4(x),
\]
the witness of Theorem~\ref{thm:aim37}, with $[s_{(1^5)}]\widetilde q=69$ against
$[s_{(5)}]\widetilde q=(4+1)^3=125$.  The two pencils are related by
$\widetilde J_i=(4J_i+2G)/5$, which is how \eqref{eq:aim36-qtilde} can be
written down from the $J_i$ and $G$ alone.

The certificate \path{counterexamples/aim-problems/verify_pencil.py} recomputes all
of this in exact rational arithmetic: it builds $V$ and the $P_i$ and checks
idempotence, $G$-self-adjointness, rank, covariance and $\sum_iP_i=3I$; forms
both pencils and certifies positive definiteness by exact leading principal
minors; expands each determinant into all $126$ monomials; derives $K$ by
enumerating semistandard Young tableaux rather than quoting it, and reads
$f^\lambda$ off its last column rather than quoting that either; inverts to get
the Schur expansions and re-extracts every coefficient independently by the
bialternant formula; verifies the closed formulas of
Proposition~\ref{prop:aim36-family} as identities in $t$; and checks that
$\widetilde q$ is Schur positive, that its lower endpoint fails at $(1^5)$ and
nowhere else, and that $t=4$ lies strictly between the two thresholds above.

% ==== end tex/04-aim36-pencil.tex ====

% ==== tex/05-macdonald-omega.tex ====
\section{Additional results for McSwiggen--Sahi Theorem 2.1}
\label{sec:macdonaldomega}

Theorem~\ref{thm:macdonald} refutes Schur-convexity at $q=t$, where Macdonald
polynomials are Schur polynomials and the arithmetic is smallest.  Nothing about
the failure needs that specialization.  This appendix records the reversal for
all $q,t\in(0,1)$, identifies the lattice points at which it is admissible,
locates the step at which the claim goes wrong, and says what survives.

Throughout, $n=2$, $\lambda=(2,0)$ and $\mu=(1,1)$, so that $\lambda\succeq\mu$,
and the evaluation point is $x=(1,1)$.  Write $t^\delta=(t,1)$.

\subsection*{The rank-two ratio}

Only two Macdonald polynomials are involved, and neither is quoted: in degree
two the triangularity $P_\lambda=m_\lambda+\text{(lower terms)}$ leaves a single
unknown, and orthogonality in the $(q,t)$ Hall inner product fixes it.

\begin{lemma}\label{lem:macdonald-rank2}
For $n=2$, $P_{(1,1)}(x_1,x_2;q,t)=x_1x_2$ and
\[
P_{(2,0)}(x_1,x_2;q,t)=x_1^2+x_2^2+A\,x_1x_2,
\qquad
A=\frac{(1+q)(1-t)}{1-qt}.
\]
At $q=t$ this is $A=1$, so $P_{(2,0)}=s_{(2)}$.
\end{lemma}
\begin{proof}
$P_{(1,1)}=e_2=m_{(1,1)}$ is forced, being the lowest partition of $2$.  Write
$P_{(2,0)}=m_{(2)}+A\,m_{(1,1)}$, and use coordinates in the power sums:
$m_{(2)}=p_2$ and $m_{(1,1)}=e_2=(p_1^2-p_2)/2$.  With
$\langle p_\lambda,p_\mu\rangle_{q,t}=\delta_{\lambda\mu}z_\lambda\prod_i(1-q^{\lambda_i})/(1-t^{\lambda_i})$
we have $\langle p_2,p_2\rangle=2(1-q^2)/(1-t^2)$,
$\langle p_1^2,p_1^2\rangle=2\bigl((1-q)/(1-t)\bigr)^2$ and
$\langle p_2,p_1^2\rangle=0$, so $\langle P_{(2,0)},P_{(1,1)}\rangle=0$ reads
\[
A\,\frac{(1-q)^2}{(1-t)^2}=(2-A)\,\frac{1-q^2}{1-t^2},
\qquad\text{i.e.}\qquad
A\bigl[(1-q)(1+t)+(1+q)(1-t)\bigr]=2(1+q)(1-t).
\]
The bracket is $2(1-qt)$, which gives $A$.
\end{proof}

\begin{proposition}\label{prop:macdonald-general}
For all $q,t\in(0,1)$,
\[
\Omega_{(2,0)}(1,1;q,t)=\frac{2+A}{1+t^2+At},
\qquad
\Omega_{(1,1)}(1,1;q,t)=\frac1t,
\]
and
\begin{equation}\label{eq:macdonald-gap}
\Omega_{(1,1)}(1,1;q,t)-\Omega_{(2,0)}(1,1;q,t)
=\frac{(1-t)^2(1-qt)}{t\,(1+t)\,(1-qt^2)}\;>\;0 .
\end{equation}
\end{proposition}
\begin{proof}
By Lemma~\ref{lem:macdonald-rank2}, $P_{(2,0)}(1,1)=2+A$ and
$P_{(2,0)}(t,1)=1+t^2+At$, while $P_{(1,1)}(1,1)=1$ and $P_{(1,1)}(t,1)=t$,
which gives the two values.  Subtracting,
\[
\frac1t-\frac{2+A}{1+t^2+At}
=\frac{1+t^2+At-t(2+A)}{t(1+t^2+At)}
=\frac{(1-t)^2}{t(1+t^2+At)},
\]
the term $At$ canceling --- so the gap does not depend on $A$ except through
the denominator.  Substituting $A$ and clearing,
\[
1+t^2+At=\frac{(1+t)(1-qt^2)}{1-qt},
\]
which turns the previous display into \eqref{eq:macdonald-gap}.  On $q,t\in(0,1)$
each of $(1-t)^2$, $1-qt$, $t$, $1+t$ and $1-qt^2$ is strictly positive.
\end{proof}

At $q=t=r$ this is $\tfrac1r-\tfrac3{1+r+r^2}=\tfrac{(1-r)^2}{r(1+r+r^2)}$, the
computation carried in Theorem~\ref{thm:macdonald}.

\subsection*{Where the point is admissible}

The reversal above is an identity in $q$ and $t$, but it refutes the claim only
where $x=(1,1)$ is a legal evaluation point, i.e.\ lies on $\mathcal L_n$.

\begin{lemma}\label{lem:macdonald-lattice}
Let $t=q^k$ with $k$ a positive integer and $a=1$.  Then
$1^n\in\mathcal L_n^{q,q^k,1}$.
\end{lemma}
\begin{proof}
Take $\mu_i=k(n-i)$, which is a partition: $\mu_i-\mu_{i+1}=k>0$.  The $i$-th
coordinate of the corresponding lattice point is
$q^{-\mu_i}t^{n-i}=q^{-k(n-i)}q^{k(n-i)}=1$.
\end{proof}

So for every positive integer $k$ the pair $\lambda=(2,0)\succeq\mu=(1,1)$ at
$x=(1,1)$ refutes Schur-convexity for the parameters $(q,q^k)$, and
Theorem~\ref{thm:macdonald} is the case $k=1$.  For an instance away from the Schur
locus take $k=2$, $q=1/2$, $t=1/4$: then $A=9/7$ and
\[
\Omega_{(2,0)}(1,1;\tfrac12,\tfrac14)=\frac{368}{155}<4=\Omega_{(1,1)}(1,1;\tfrac12,\tfrac14),
\]
a gap of $252/155$.

\subsection*{Where the claim goes wrong}

The failure is not an accident of small rank: it is the principal-specialization
normalization, which does not commute with the determinant shift.

\begin{lemma}[Determinant shift for $\Omega$]\label{lem:macdonald-shift}
Let $\lambda=(\lambda_1,\ldots,\lambda_N)$ and $c=\lambda_N$.  Then
\[
\Omega^{(N)}_\lambda(y;q,t)
=\left(\frac{y_1\cdots y_N}{t^{N(N-1)/2}}\right)^{c}
\Omega^{(N)}_{\lambda-c(1^N)}(y;q,t).
\]
\end{lemma}
\begin{proof}
From $P_\lambda(y)=(y_1\cdots y_N)^cP_{\lambda-c(1^N)}(y)$, evaluation at
$t^\delta$ contributes $\bigl(\prod_{i=1}^Nt^{N-i}\bigr)^c=t^{cN(N-1)/2}$, since
$\sum_{i=1}^N(N-i)=N(N-1)/2$.  Divide the two relations.
\end{proof}

The factor $t^{-cN(N-1)/2}$ is exactly what the argument for Schur-convexity
drops.  Its effect is not a perturbation.  At $N=2$ and $\lambda=(1,1)$, where
$c=1$ and $\lambda-c(1^2)=(0,0)$, the lemma gives
$\Omega_{(1,1)}(y)=y_1y_2/t$, whereas the shift without the $t$-power would give
$y_1y_2$; the two agree only at $t=1$.  Evaluated at $y=1^2$ this is the $1/t$
of Proposition~\ref{prop:macdonald-general}, which diverges as $t\downarrow0$ while
$\Omega_{(2,0)}(1,1;q,t)\to3+q$ stays bounded.  The reversal is therefore not
delicate: past a point the two sides are not even the same order of magnitude.

\subsection*{What is not refuted}

Normalizing at $1^n$ instead of at $t^\delta$ removes the phenomenon entirely.
Write $W_\lambda(x;q,t)=P_\lambda(x;q,t)/P_\lambda(1^n;q,t)$.  Since
$P_\lambda(1^N)=P_{\lambda-c(1^N)}(1^N)$, the determinant shift for $W$ carries
no extra factor at all, and $W_\lambda(1^n;q,t)=1$ for every $\lambda$ --- so the
witness above says nothing whatever about $W$, both sides being $1$ at the
evaluation point.  More than that, in the same rank-two case the $W$-inequality
holds where the $\Omega$-inequality fails:
\[
W_{(2,0)}(x_1,x_2;q,t)-W_{(1,1)}(x_1,x_2;q,t)
=\frac{x_1^2+x_2^2+Ax_1x_2-(2+A)x_1x_2}{2+A}
=\frac{(x_1-x_2)^2}{2+A}\;\ge\;0 .
\]
This is the rank-two shadow of a general fact on the Schur locus $q=t$, where
$W_\lambda$ is the normalized Schur function $s_\lambda(x)/s_\lambda(1^n)$: for
every $x$ with nonnegative entries, $\lambda\succeq\mu$ implies
$s_\lambda(x)/s_\lambda(1^n)\ge s_\mu(x)/s_\mu(1^n)$, which is the conjecture of
Cuttler, Greene and Skandera, proved by \citet{SraNormalizedSchur2016}.  So on
that locus the $1^n$-normalized ratio is Schur-convex in exactly the sense that
$\Omega_\lambda$ is not.
Theorem~\ref{thm:macdonald} is thus a statement about the normalization, not about
Macdonald polynomials on the lattice: what it refutes is Schur-convexity of
$\lambda\mapsto\Omega_\lambda$, and any corresponding claim for $W_\lambda$ is
untouched.

The certificate \path{counterexamples/macdonald-schur-convexity/verify.py}
recomputes all of this exactly: it recovers $A$ from the $(q,t)$ Hall inner
product rather than quoting Lemma~\ref{lem:macdonald-rank2}, verifies
\eqref{eq:macdonald-gap} as a rational-function identity in $q$ and $t$ together
with the factorization of $1+t^2+At$ that makes each sign visible, checks the
lattice membership of Lemma~\ref{lem:macdonald-lattice} symbolically in $q$, $k$ and
$n$, checks the shift exponent and the $N=2$ instance where the naive shift
fails, and confirms the $W$ identity above.

% ==== end tex/05-macdonald-omega.tex ====

% ==== tex/06-quantum-coupon-trace.tex ====
\section{Some refinements of the quantum coupon collection conjecture}
\label{sec:qcctrace}

Theorem~\ref{thm:quantum-coupon-collector} refutes the conjecture in the Loewner order,
which is what was asked.  The failure is not confined to that ordering.  Write
\[
\Delta_n(X_1,\ldots,X_n):=\tr Q_n(X_1,\ldots,X_n)
=\sum_{\emptyset\ne S\subseteq[n]}(-1)^{|S|-1}\tr X_S^{-1},
\]
with $X_S:=\sum_{i\in S}X_i$ as in the case.  A positive definite matrix has
positive trace, so the conjecture implies $\Delta_n>0$, and
Theorem~\ref{thm:qcc-positive-five} gives $\Delta_n>0$ for every $d$ and every $n\le5$.
This scalar consequence fails too, and for a reason one can see in closed form.

Let $u_1,\ldots,u_n\in\R^3$ be such that every pair is linearly independent and every triple spans $\R^3$, and put
\[
X_i(\varepsilon):=u_iu_i^{\mathsf T}+\varepsilon I_3,\qquad\varepsilon>0,
\]
so that $X_S(\varepsilon)=\sum_{i\in S}u_iu_i^{\mathsf T}+k\varepsilon I_3$ for $|S|=k$.  A singleton leaves two regularized null directions, so $\tr X_S(\varepsilon)^{-1}=2/\varepsilon+O(1)$; a pair leaves one, so $\tr X_S(\varepsilon)^{-1}=1/(2\varepsilon)+O(1)$; and for $k\ge3$ the unregularized matrix is already positive definite, so the trace of the inverse is $O(1)$.  Only finitely many subsets occur, so the remainders sum uniformly and
\begin{equation}\label{eq:qcc-asymptotic}
\Delta_n\bigl(X_1(\varepsilon),\ldots,X_n(\varepsilon)\bigr)
=\frac1\varepsilon\Bigl(2n-\tfrac12\binom n2\Bigr)+O(1)
=\frac{n(9-n)}{4\varepsilon}+O(1).
\end{equation}
The coefficient is negative exactly for $n\ge10$, which the next theorem realizes with an exact witness.

\begin{theorem}[\statusfalse: the trace consequence]\label{thm:quantum-coupon-collector-trace}
There are $d=3$, $n=10$ and $X_1,\ldots,X_{10}\in\mathbf{S}_{++}^{3}$ with $\Delta_{10}(X_1,\ldots,X_{10})<0$.  A fortiori $Q_{10}\succ0$ fails.
\end{theorem}
\begin{proof}
Take $d=3$, $n=10$, $\varepsilon=1/10000$, and, for $t=0,1,\ldots,9$,
\[
u_t=(1,t,t^2)^{\mathsf T},\qquad X_{t+1}=u_tu_t^{\mathsf T}+\varepsilon I_3.
\]
Every $X_i$ is strictly positive definite, and every three of the $u_t$ are independent because their determinant is a nonzero Vandermonde product; the matrices are noncommuting, the $(1,2)$ entry of $[X_1,X_2]$ being $1$.  Exact rational evaluation of all $2^{10}-1=1023$ nonempty-subset terms gives
\[
-13901<\Delta_{10}(X_1,\ldots,X_{10})<-13900.
\]
The supplied verifier computes this with exact rational arithmetic and records the full numerator and denominator in \texttt{artifacts/certificate.json}.
\end{proof}

\begin{remark}[Scope]
The record for the trace form is coarser than for the conjecture itself: true
for $n\le5$ by Theorem~\ref{thm:qcc-positive-five}, false at $n=10$, and undecided for
$6\le n\le9$.  Equation \eqref{eq:qcc-asymptotic}, whose coefficient vanishes at
$n=9$, says nothing there, since that construction is only one family; and a
$Q_n$ that is not positive semidefinite, such as the one of
Theorem~\ref{thm:quantum-coupon-collector}, can still have positive trace.
\end{remark}

% ==== end tex/06-quantum-coupon-trace.tex ====

% ==== tex/08-lorentzian-audit.tex ====
\section{Two new results for hyperbolic polynomials (author: S.~Sra)}
\label{sec:lorentzianaudit}
\noindent\dblue{\emph{These results are non-AI (by S.~Sra); including here to share with interested folks.}}

\vskip8pt

Replacing ``Lorentzian'' by ``hyperbolic'' turns the refuted principle into a
theorem.  Let $p:\R^n\to\R$ be homogeneous of degree $m\ge1$ and hyperbolic with
respect to $e$, normalized so that $p(e)>0$, and let $\Lambda_{++}$ be the open
hyperbolicity cone containing $e$.  By \citet{Garding1959},
$\Lambda_{++}$ is an open convex cone, $p$ is hyperbolic with respect to each of
its points, and $F:=p^{1/m}$ is positive, homogeneous of degree one and concave
on it.  The other external input is the Helton--Vinnikov theorem
\citep{HeltonVinnikov2007,LewisParriloRamana2005}, in the normalized form: a
ternary form $q$ of degree $m$, hyperbolic with respect to $\mathbf 1=(1,1,1)$
with $q(\mathbf 1)>0$, admits real symmetric $A_1,A_2,A_3$ with
$A_1+A_2+A_3=I_m$ and $q(r)=q(\mathbf 1)\det(r_1A_1+r_2A_2+r_3A_3)$.  Three
points at a time is exactly what a triangle inequality needs.

\begin{lemma}[Ternary restriction]\label{lem:lor-ternary}
Fix $x_1,x_2,x_3\in\Lambda_{++}$, put $T(r):=r_1x_1+r_2x_2+r_3x_3$ and
$q:=p\circ T$.  Then $q$ is homogeneous of degree $m$ and hyperbolic with
respect to $\mathbf 1$, and each coordinate vector $\varepsilon_i\in\R^3$ lies in
the hyperbolicity cone of $q$ containing $\mathbf 1$.
\end{lemma}
\begin{proof}
Homogeneity is immediate, and $w:=T(\mathbf 1)=x_1+x_2+x_3\in\Lambda_{++}$
because the cone is convex.  So $p$ is hyperbolic with respect to $w$, and for
every $r\in\R^3$ the univariate $\lambda\mapsto q(r+\lambda\mathbf 1)
=p(T(r)+\lambda w)$ has only real zeros; that is hyperbolicity with respect to
$\mathbf 1$.  For the second claim, the path
$\gamma_i(t)=(1-t)\varepsilon_i+t\mathbf 1$ satisfies
$T(\gamma_i(t))=x_i+t\sum_{j\ne i}x_j\in\Lambda_{++}$, so $q(\gamma_i(t))>0$ on
$[0,1]$: it joins $\varepsilon_i$ to $\mathbf 1$ inside the positive component.
\end{proof}

\begin{lemma}[Positive definite representatives]\label{lem:lor-pdrep}
In that setting the matrices $A_1,A_2,A_3$ supplied by Helton--Vinnikov are
positive definite, and for all $i,j$,
\[
p(x_i)=q(\mathbf 1)\det A_i,\qquad
p\Bigl(\frac{x_i+x_j}2\Bigr)=q(\mathbf 1)\det\Bigl(\frac{A_i+A_j}2\Bigr).
\]
\end{lemma}
\begin{proof}
Along the path of Lemma~\ref{lem:lor-ternary},
$L_i(t):=\sum_k\gamma_i(t)_kA_k$ satisfies
$q(\gamma_i(t))=q(\mathbf 1)\det L_i(t)>0$, so $L_i(t)$ is nonsingular on
$[0,1]$; the eigenvalues of a continuous path of nonsingular real symmetric
matrices cannot change sign, so the inertia is constant.  At $t=1$,
$L_i(1)=A_1+A_2+A_3=I_m\succ0$, hence $A_i=L_i(0)\succ0$.  The two identities
are the determinantal representation evaluated at $\varepsilon_i$ and at
$(\varepsilon_i+\varepsilon_j)/2$.
\end{proof}

\begin{theorem}[\statusproved: hyperbolic Jensen--Stein pseudometric]\label{thm:lor-hyperbolic}
Let $p$ be hyperbolic as above.  Then $\sqrt{\J_p}$, with
$\J_p(x,y)=\log p\bigl(\frac{x+y}2\bigr)-\frac12\log p(x)-\frac12\log p(y)$, is a
pseudometric on $\Lambda_{++}$.
\end{theorem}
\begin{proof}
Concavity of $F=p^{1/m}$ gives
$F(\frac{x+y}2)\ge\frac{F(x)+F(y)}2\ge\sqrt{F(x)F(y)}$, which after raising to
the $m$th power makes $\J_p\ge0$; symmetry and vanishing on the diagonal are
immediate.  For the triangle inequality fix $x_1,x_2,x_3\in\Lambda_{++}$ and take
$A_1,A_2,A_3\succ0$ from Lemma~\ref{lem:lor-pdrep}.  The factor $q(\mathbf 1)$
cancels, leaving
\[
\J_p(x_i,x_j)
=\log\frac{\det\bigl((A_i+A_j)/2\bigr)}{\sqrt{\det A_i\det A_j}}
=\delta_{\mathrm S}^2(A_i,A_j),
\]
the squared Stein divergence, whose square root is a metric on the positive
definite cone \citep{SraSdiv}.  Applying that triangle inequality to
$A_1,A_2,A_3$ gives the claim, and the triple was arbitrary.
\end{proof}

\noindent Theorem~\ref{thm:lor-fourpoint} below proves a conjecture once shared with the author (S.~Sra) by~\citet{thomas2025}.
\begin{theorem}[\statusproved]\label{thm:lor-fourpoint}
With $F=p^{1/m}$, for all $a,b,c$ in the closed hyperbolicity cone,
\[
F(a+b)\,F(a+c)\;\ge\;F(a)\,F(a+b+c)+F(b)\,F(c).
\]
\end{theorem}
\begin{proof}
Write $\mathcal D(M)=\det(M)^{1/m}$.  We first prove the matrix form
$\mathcal D(A+B)\mathcal D(A+C)\ge\mathcal D(A)\mathcal D(A+B+C)+\mathcal D(B)\mathcal D(C)$
for $A,B,C\succeq0$.  For $A\succ0$, put $X=A^{-1/2}BA^{-1/2}$ and
$Y=A^{-1/2}CA^{-1/2}$; every term carries $\mathcal D(A)^2$, so it suffices to
show $\mathcal D(I+X)\mathcal D(I+Y)\ge\mathcal D(I+X+Y)+\mathcal D(X)\mathcal D(Y)$.
Order the eigenvalues of $X$ decreasingly and those of $Y$ increasingly.  Since a
product does not depend on the ordering,
$\mathcal D(I+X)\mathcal D(I+Y)=\bigl(\prod_i[(1+x_i+y_i)+x_iy_i]\bigr)^{1/m}$, and
Minkowski's determinant inequality
$\det(R+S)^{1/m}\ge\det R^{1/m}+\det S^{1/m}$ --- itself Jensen's inequality for
$t\mapsto\log(1+e^t)$ applied to the eigenvalues of $R^{-1/2}SR^{-1/2}$ ---
applied to the diagonal matrices $R=\diag(1+x_i+y_i)$ and $S=\diag(x_iy_i)$
gives
\[
\mathcal D(I+X)\mathcal D(I+Y)\ \ge\ \Bigl(\prod_i(1+x_i+y_i)\Bigr)^{1/m}
+\Bigl(\prod_ix_iy_i\Bigr)^{1/m}.
\]
Fiedler's upper bound \citep{Fiedler1971}, which pairs one decreasing spectrum
with the other increasing one, gives $\det(I+X+Y)\le\prod_i(1+x_i+y_i)$, while
the second term is $\mathcal D(X)\mathcal D(Y)$; a limit $\varepsilon\downarrow0$
in $A+\varepsilon I$ covers singular $A$.  For the cone statement, apply
Lemmas~\ref{lem:lor-ternary} and~\ref{lem:lor-pdrep} to $a,b,c\in\Lambda_{++}$ to
get $A,B,C\succ0$ with $p(ra+sb+tc)=q(\mathbf 1)\det(rA+sB+tC)$; then
$F(ra+sb+tc)=q(\mathbf 1)^{1/m}\mathcal D(rA+sB+tC)$ for $r,s,t\ge0$, and the
displayed inequality is the matrix one times $q(\mathbf 1)^{2/m}$.  Points of the
closed cone are handled by $u\mapsto u+\varepsilon e$, which lies in
$\Lambda_{++}$, and continuity.
\end{proof}

Nothing here collides with Theorem~\ref{thm:volume-distance} and
Corollary~\ref{thm:lorentzian}, and it is worth saying why not.  The witness
\eqref{eq:lorentzian-cubic} is strictly Lorentzian but \emph{not} hyperbolic:
$G(x,1)=\frac1{10}x^3+\frac32x^2+\frac{11}5x+1$ has discriminant
\[
18abcd-4b^3d+b^2c^2-4ac^3-27a^2d^2=-\frac{1499}{1250}<0
\qquad
\Bigl(a,b,c,d=\tfrac1{10},\tfrac32,\tfrac{11}5,1\Bigr),
\]
so it has one real and two complex conjugate roots.  A binary form is hyperbolic
with respect to some direction only if it splits into real linear factors --- each
factor $\ell_i$ contributes the root $-\ell_i(x)/\ell_i(e)$, which is real for
every real $x$ only when $\ell_i$ is real up to scale --- so $G$ is hyperbolic
with respect to no direction at all.  In two variables hyperbolicity is therefore
strictly stronger than strict Lorentzianity, and the witness lives in the gap.
That gap is the content of both results together: what makes the midpoint Jensen
gap a squared distance is the determinantal representation that hyperbolicity
buys through Helton--Vinnikov, and not log-concavity of the coefficients.

% ==== end tex/08-lorentzian-audit.tex ====

\cxendrule   %
\vskip 12pt

\end{document}